\documentclass[11pt]{amsart}
\usepackage[utf8]{inputenc}
\usepackage{amssymb}
\usepackage{graphicx}
\usepackage{color}
\usepackage{amsmath}
\usepackage{axodraw4j}
\usepackage{mathrsfs}
\usepackage{amscd}
\usepackage{epic}
\usepackage{color}
\usepackage[all]{xy}
\usepackage{tikz}
\usepackage{comment}
\newtheorem{deft}{D\'efinition}
\newtheorem{theorem}{Theorem}[section]

\newtheorem{remark}[theorem]{Remark}

\newtheorem{example}[theorem]{Example}

\numberwithin{equation}{section}

\font \sevenrm=cmr7
\font \fiverm=cmr5

\newcommand{\smop}[1]{\mathop{\hbox {\sevenrm #1} }\nolimits}
\newcommand{\ssmop}[1]{\mathop{\hbox {\fiverm #1} }\nolimits}
\def \restr#1{\mathstrut_{\textstyle |}\raise-6pt\hbox{$\scriptstyle #1$}}
\def \srestr#1{\mathstrut_{\scriptstyle |}\hbox to
-1.5pt{}\raise-4pt\hbox{$\hskip 1pt\scriptscriptstyle #1$}}

\def\shuffle{\sqcup\mathchoice{\mkern-7mu}{\mkern-7mu}{\mkern-3.2mu}{\mkern-3.8mu}\sqcup}

\def\L1#1{L^1(#1)}

\def\L#1#2{L^{#1}(#2)}

\def\lef({\left(}
\def\rig){\right)}

\newcommand{\tp}{\mathbf{p}} 
\newcommand{\tq}{\mathbf{q}}

\newcommand{\mop}[1]{\mathop{\hbox {\rm #1} }}
\newcommand{\dm}[1]{\textcolor{blue}{D: #1}}
\newcommand{\dmr}[1]{\textcolor{red}{D: #1}}

\def\entoure #1{\scalebox{0.8}{
  \begin{picture}(12,10) (0,0)
    \SetWidth{1.0}
    \SetColor{Black}
    \Arc(6,3)(6,90,450)
    \Text(3,0.5)[lb]{\Black{$#1$}}
  \end{picture}}
}
\def\stick #1#2{\scalebox{0.8}{
  \begin{picture}(12,15) (0,0)
    \SetWidth{1.0}
    \SetColor{Black}
    \Arc(6,-8)(6,90,450)
    \Text(2.5,-11.5)[lb]{\Black{$#1$}}
    \Arc(6,12)(6,90,450)
    \Text(2.5,8.5)[lb]{{\Black{$#2$}}}
    \Line(6,6)(6,-2)
  \end{picture}}
}
\def\ladder #1#2#3{\scalebox{0.8}{
  \begin{picture}(12,15) (0,0)
    \SetWidth{1.0}
    \SetColor{Black}
    \Arc(6,-8)(6,90,450)
    \Text(2.5,-11.5)[lb]{\Black{$#1$}}
    \Arc(6,12)(6,90,450)
    \Text(2.5,8.5)[lb]{{\Black{$#2$}}}
    \Arc(6,32)(6,90,450)
    \Text(2.5,28.5)[lb]{{\Black{$#3$}}}
    \Line(6,6)(6,-2)
    \Line(6,26)(6,18)
  \end{picture}}
}
\def\cherry #1#2#3{\scalebox{0.8}{
  \begin{picture}(12,20) (0,0)
    \SetWidth{1.0}
    \SetColor{Black}
    \Arc(6,-8)(6,90,450)
    \Text(2.5,-11.5)[lb]{\Black{$#1$}}
    \Arc(-6,12)(6,90,450)
    \Text(-10,8.5)[lb]{{\Black{$#2$}}}
    \Line(-4,7)(3,-2)
    \Arc(18,12)(6,90,450)
    \Text(15,8.5)[lb]{{\Black{$#3$}}}
    \Line(16,7)(9,-2)
  \end{picture}}
}
\def\igrec #1#2#3#4{\scalebox{0.8}{
  \begin{picture}(12,0) (0,-20)
    \SetWidth{1.0}
    \SetColor{Black}
    \Arc(6,-28)(6,90,450)
    \Line(6,-22)(6,-14)
    \Text(2.5,-31.5)[lb]{\Black{$#1$}}
    \Arc(6,-8)(6,90,450)
    \Text(2.5,-11.5)[lb]{\Black{$#2$}}
    \Arc(-6,12)(6,90,450)
    \Text(-10,8.5)[lb]{{\Black{$#3$}}}
    \Line(-4,7)(3,-2)
    \Arc(18,12)(6,90,450)
    \Text(15,8.5)[lb]{{\Black{$#4$}}}
    \Line(16,7)(9,-2)
  \end{picture}}
}

\def\treeone #1#2#3#4#5#6{\scalebox{0.8}{
  \begin{picture}(12,0) (0,-20)
    \SetWidth{1.0}
    \SetColor{Black}
    \Arc(6,-28)(6,90,450)
    \Line(6,-22)(6,-14)
    \Text(2.5,-31.5)[lb]{\Black{$#1$}}
    \Arc(6,-8)(6,90,450)
    \Text(2.5,-11.5)[lb]{\Black{$#2$}}
    \Arc(-9,12)(6,90,450)
    \Text(-13,8.5)[lb]{{\Black{$#3$}}}
    \Line(-7,7)(3,-2)
    \Arc(21,12)(6,90,450)
    \Text(18,8.5)[lb]{{\Black{$#4$}}}
    \Line(19,7)(9,-2)
    \Arc(6,12)(6,90,450)
    \Text(2.5,8.5)[lb]{{\Black{$#5$}}}
    \Line(6,6)(6,-2)
    \Arc(6,32)(6,90,450)
    \Text(2.5,28.5)[lb]{{\Black{$#6$}}}
    \Line(6,26)(6,18)
  \end{picture}}
}

\def\treeones #1#2#3#4#5#6{\scalebox{0.8}{
  \begin{picture}(12,0) (0,-20)
    \SetWidth{1.0}
    \SetColor{Black}
    \Arc(6,-28)(6,90,450)
    \Line(6,-22)(6,-14)
    \Text(2.5,-31.5)[lb]{\Black{$#1$}}
    \Arc(6,-8)(6,90,450)
    \Text(2.5,-11.5)[lb]{\Black{$#2$}}
    \Arc(-9,12)(6,90,450)
    \Text(-13,8.5)[lb]{{\Black{$#3$}}}
    \Line(-7,7)(3,-2)
    \Arc(-9,32)(6,90,450)
    \Text(-13,28.5)[lb]{{\Black{$#6$}}}
    \Line (-9,26)(-9,18)
    \Arc(21,12)(6,90,450)
    \Text(18,8.5)[lb]{{\Black{$#4$}}}
    \Line(19,7)(9,-2)
    \Arc(6,12)(6,90,450)
    \Text(2.5,8.5)[lb]{{\Black{$#5$}}}
    \Line(6,6)(6,-2)
  \end{picture}}
}
\begin{document}
 \title{ Operad structures on the species composition of two operads}
 \author{Imen Rjaiba}
 \address{} \email{}
 \address{} \email{}
 \maketitle{}
  \begin{abstract}
     We give an explicit description of three operad structures on the species composition $\tp \circ \tq$, where $\tq$ is any given positive operad, and where $\tp$  is the $\mop{NAP}$ operad, or a shuffle version of the magmatic operad $\mop{\sc Mag}$. No distributive law between $\tp$ and $\tq$ is assumed. 
  \end{abstract}

  \section{\bf{Introduction.}}\label{sec:1}
  The notion of operad was introduced by J. P. May in 1972 in the context of algebraic topology to understand spaces of iterated loops \cite{M1972}. Operads and their reformulations in purely algebraic frameworks \cite{S1986, GJ1994, Lo1996, Ma1996} provoked a revival of interest in the Nineties of last century. A. Joyal in \cite{J1981, J1986} gave a very important formulation of operads in the species formalism, which has been developed by many others such as F. Bergeron, G. Labelle and P. Leroux \cite{BLL1998}, M. M\'endez \cite{M2015} as well as M. Aguiar and S. Mahajan \cite{AM2010, AM2020}, and others.

  There are two operad structures for non-planar rooted tree species: in 2001, F. Chapoton and M. Livernet introduced the notion of $\mop{\sc Pre-Lie}$ operad, then M. Livernet defined in 2006 the $\mop{NAP}$ operad. The free operad generated by one single binary operation is called the magmatic operad, and is traditionnally described in terms of planar binary trees with labelled leaves. In 2012, K. Ebrahimi-Fard and D. Manchon \cite{ED2014} used D. Knuth's rotation correspondence (see \cite{Knuth}) in order to describe the magmatic operad in the planar rooted tree formalism. The present article focuses mainly on these three operads.\\

  This paper is organized as follows: the second section is devoted to introductory background
  on species and operads. The third section is a reminder of  rooted tree notion and of $\mop{NAP}$ and $\mop{\sc Pre-Lie}$ operads. 

The fourth section is devoted to the construction of explicit partial compositions $\square_s$ providing an operad structure on the species composition $\mop{NAP} \circ \, \tq$. The fifth section contains three subsections: in \ref{sec 6.1} we remind the notion of planar binary tree and Knuth’s rotation correspondence between planar binary trees and planar rooted trees. The $\mop{\sc Mag}$ operad is reminded in \ref{sec 6.2}, and Paragraph \ref{sec 6.3} is devoted to a new operad structure on $\mop{\sc Mag}$ with partial compositions denoted by $\triangle_s$, derived from the magmatic operad structure by shuffling the branches.
In the last section we use this new operadic structure on $\mop{\sc Mag}$ to define an operad structure $\diamond$ on $\mop{\sc Mag}\circ \tq$.

\section{Preliminaries}\label{sec:2}
    \subsection{Species}
    A species is a functor from the category of finite sets $\mop{Fin}$ to the category of vector species $\mop{Vect}$.
    We write $\mop{Sp}$ for the category of species. Given a species $\tp$, we denote the image of a finite set I by $\tp [I]$.
    Hence, a species consists of a family of vector spaces $\tp[I]$, one for each finite set $I$, and linear maps:
    $$\tp [\sigma]: \tp[I] \to \tp[J],$$
    for each bijection $\sigma: I \to J$, such that $\tp[id_I]=id_{\tp[I]}$ and $\tp [\theta\sigma]= \tp[\theta]  \tp[\sigma]$, where $\sigma: I \to J$ and $\theta: J \to K$ are two composable bijections. Therefore, the map $\tp[\sigma]$ is an isomorphism whose inverse is $\tp[\sigma ^{-1}]$.
    Similarly, a species morphism $f : \tp \to \tq$ consists of a family of linear maps defined as follows, for each finite set $I$:
    $f_I : \tp[I] \to \tq[I]$, such that for each bijection $\sigma : I \to J$, the following diagram commutes: 
   
       $$ \xymatrix{
        \tp[I] \ar[d]_{\tp[\sigma]} \ar[r]^{f_I} & \tq[I] \ar[d]^{\tq[\sigma]} \\
        \tp[J] \ar[r]_{f_J} & \tq[J] }$$
    
  Note that a species $\tp$ is positive if $\tp[\emptyset] = \{0\}$.\\
  \textbf{\textit{*Species composition:}}
  Let $\tp$ a species and $\tq$ a positive species. Define the composition of two species $\tp \circ \tq$ as follows. For any finite set $I$,
$$\tp \circ \tq[I] :=\displaystyle \bigoplus_{X\vdash I}\tp[X] \otimes \tq(X),$$
where for a given partition $X$ of $I$, $$\tq(X):=\displaystyle\bigotimes_{S\in X} \tq[S].$$
The unit for the composition is the species $\mathbb I$ is defined by $\mathbb I[I]=\mathbb K$ for $|I|=1$, and $\mathbb I[I]=\{0\}$ otherwise. 
    \subsection{Operads}
  \subsubsection{Global definition of a symmetric operad}
  A symmetric operad is a monoid in $(\mop{Sp}_{+}, \circ)$. More explicitly, a positive symmetric operad is a positive species $\tq$ with associative and unital morphisms of species 
  \[\mu: \tq \circ \tq \to \tq \qquad \text{and} \qquad \eta: \mathbb I \to \tq.\]
  Therefore, for each nonempty finite set $I$ and partition $X \vdash I$, there is a linear map 
  $$\mu: \tq[X] \otimes \left( \displaystyle \bigotimes_{S \in X}\tq[S]\right)\to \tq[I]$$
  called operadic composition, and for each singleton $\{\star\}$, there a unit map, 
  \[\eta : \mathbb{K} \to \tq[\{\star\}]\]
  satisfying naturality, associativity and unitality axioms.
  \subsubsection{Definition of a symmetric operad through partial compositions}
  A symmetric operad is a positive species $\tq$ such that, for any two nonempty finite sets $S,T$, and for any singleton $\{\star\}$, there exist  elements $ u_{\star} \in \tq[\{\star\}]$ and a partial composition operation \[\circ_{s} : \tq[S]\otimes \tq[T] \to \tq[S \sqcup_{s}T],\] 
    satisfying the followings axioms:
  		\begin{itemize}
 		\item Associativity: for $ x \in \tq[S]$, $y \in \tq[T]$
  				\begin{itemize}
  					\item[(A1)] $(x \circ_{s}y)\circ_{s'}z=(x \circ_{s'}z)\circ_{s}y$, if $s$ and $s'$ are two distinct elements of $S$. (Parallel associativity)
  					\item[(A2)] $(x \circ_{s}y)\circ_{t}z=x \circ_{s}(y\circ_{t}z)$, if $s \in S$ and $t \in T$. (Nested associativity)
  				\end{itemize}
      \vspace{.1in}
  		\item Naturality: given two bijections $\sigma_{1}:S \to S'$ and $\sigma_{2}:T \to T'$ and $s \in S$. Then, for every $x \in \tq[S]$, $y \in \tq[T]$
  		\begin{itemize}
  			\item[(N1)] $\tq[\sigma_{1}](x) \circ_{\sigma(s)} \tq[\sigma_{2}](y)= \tq[\sigma](x \circ_{s}y)$, where
       \[\sigma:=(\sigma_{1})|_{S\setminus\{s\}} \bigcup \sigma_{2}:S \,{\sqcup}_{s} T \to S' \sqcup_{\sigma_{1}(s)} T'.\]
  		\item[(N2)] If $s_{1},s_{2}\in S$, then for any bijection $\sigma:\{s_{1}\}\to \{s_{2}\}$ we have $$\tq[\sigma](u_{s_{1}})=u_{s_{2}}.$$
  				\end{itemize}
  			\item Unitality: for $s \in S$ and $x \in \tq[S]$, we have
  				\begin{itemize}
  					\item[(U1)] $u_{\star}\circ_{\star}x=x$, for any singleton $\{\star\}$;
  					\item[(U2)] $x \circ_{s}u_{s}=x.$
  				\end{itemize}
  		\end{itemize}
  \section{\bf{The operads $\mop{NAP}$ and $\mop{\sc Pre-Lie}$ }}
  \subsection{Rooted trees}
  A rooted tree is a nonempty connected graph without loop together with a distinguished vertex called the root.
  The species of rooted trees is denoted by $\mathcal{T}$.
  \[\mathcal{T}[\{\emptyset\}]=\{0\}, \hskip 4mm\mathcal{T}[\{*\}]=\left<\{\!\entoure *\}\right>,\hskip 4mm \mathcal{T}[\{1,2\}]=\left<\left\{\stick 12,\ \stick 21\right\}\right>,\]
  \vskip 3mm
  \[\mathcal T[\{1,2,3\}]=
  \left<\ladder 123, \, \ladder 213,\,  \ladder 132,\,  \ladder 312, \, \ladder 231,\,  \ladder 321, \hskip 4mm \cherry 123\hskip 2mm , \hskip 5mm \cherry 231\hskip 2mm , \hskip 5mm \cherry 312\hskip 5mm\right>.\]
  It is well-known that $\mop{dim}\mathcal{T}[\{1,2,...,n\}]=n^{n-1}$ (see \cite{J1981}, example 12). The species $\mathcal{T}$ is endowed with two different operad structures: the $\mop{NAP}$ operad introduced by M. Livernet in \cite{L2006} and the $\mop{\sc Pre-Lie}$ operad introduced by F. Chapoton and M. Livernet in \cite{CL2001}.\\
  
  In what follows, we use $\mop{In}(s,v)$ to denote the set of edges of a tree $v$ arriving at vertex $s$ and $\mop{Ver}(v)$ to denote the set of vertices of a tree $v$.
  \subsection{The operad structure $\mop{NAP}$}
  Let $I$ be a finite set, $I=S \sqcup T$, and $s \in S$. We define the partial composition as follows: $$\circ_{s}:\mathcal{T}[S] \otimes \mathcal{T}[T] \longrightarrow \mathcal{T}[S \sqcup_{s} T],$$ such that for $u \in \mathcal{T}[S]$ and $v \in \mathcal{T}[T]$, the partial composition $u \circ_{s} v$ is the tree of $\mathcal{T}[S \sqcup_{s} T]$ obtained by replacing the vertex $s$ of $u$ by the tree $v$ and connecting each edge of $u$ arriving at vertex $s$ of $u$ at the root of $v$. For example,
  \vskip 4mm
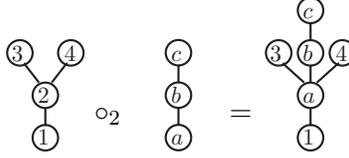
  \[\igrec 1234\hskip 4mm\circ_2\hskip 4mm\ladder abc \hskip 4mm =\hskip 4mm \treeone 1a34bc \]
  \vskip 3mm
  \noindent The {\sc NAP} operad verifies an important additional property, namely
  \begin{equation}\label{asso-suppl}
    (t\circ_s u)\circ_{\smop{root}(u)} v=\mathcal T[\varphi_{uv}]\big((t\circ_s v)\circ_{\smop{root}(v)} u\big)
  \end{equation}
  for any vertex $s$ of $t$, where $\varphi_{uv}$ is the bijective map from  $\mop{Ver}(t)\sqcup_s \mop{Ver}(v)\sqcup_{\smop{root}(v)} \mop{Ver}(u)$ onto $\mop{Ver}(t)\sqcup_s \mop{Ver}(u)\sqcup_{\smop{root}(u)} \mop{Ver}(v)$ which replaces the root of $u$ by the root of $v$.
  \subsection{The operad structure $\mop{\sc Pre-Lie}$}
  Let $I$ be a finite set, $I=S \sqcup T$, and $s \in S$. We define the partial composition as follows:
  $$\circ_{s}:\mathcal{T}[S] \otimes \mathcal{T}[T]\longrightarrow \mathcal{T}[S \sqcup_{s} T],$$ such that for $u \in \mathcal{T}[S]$ and $v \in \mathcal{T}[T]$, the partial composition
  $$u \circ_{s} v= \displaystyle \sum_{f: \smop{In}(s,u) \to \smop{Ver}(v)}u \circ_{s}^{f} v$$ 
  is the tree of $\mathcal{T}[S \sqcup_{s} T]$ obtained by replacing the vertex $s$ of $u$ by the tree $v$ and
connecting each edge $a\in\mop{In}(s,u)$ to the vertex $f(a)$ of $v$. For example,
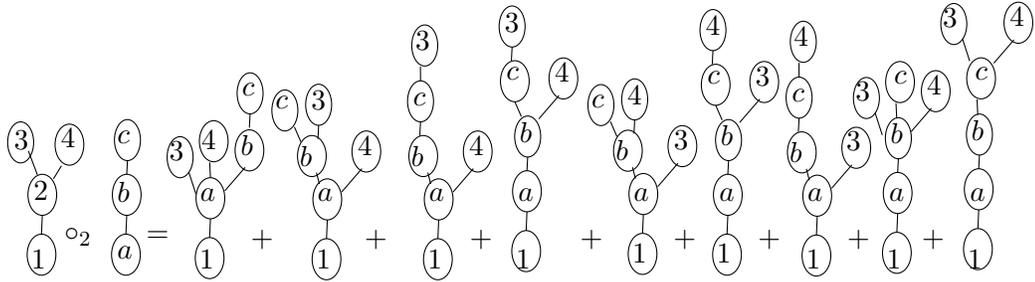
\begin{center}
    \begin{tikzpicture}[x=0.75pt,y=0.75pt,yscale=-1,xscale=1]

\draw   (19.83,293.84) .. controls (19.83,288.12) and (23.07,283.48) .. (27.07,283.48) .. controls (31.07,283.48) and (34.31,288.12) .. (34.31,293.84) .. controls (34.31,299.56) and (31.07,304.2) .. (27.07,304.2) .. controls (23.07,304.2) and (19.83,299.56) .. (19.83,293.84) -- cycle ;
\draw   (10,237.17) .. controls (10,231.45) and (12.93,226.81) .. (16.55,226.81) .. controls (20.17,226.81) and (23.1,231.45) .. (23.1,237.17) .. controls (23.1,242.89) and (20.17,247.52) .. (16.55,247.52) .. controls (12.93,247.52) and (10,242.89) .. (10,237.17) -- cycle ;
\draw   (33.91,238.32) .. controls (33.91,232.6) and (37.15,227.96) .. (41.15,227.96) .. controls (45.15,227.96) and (48.4,232.6) .. (48.4,238.32) .. controls (48.4,244.04) and (45.15,248.67) .. (41.15,248.67) .. controls (37.15,248.67) and (33.91,244.04) .. (33.91,238.32) -- cycle ;
\draw   (20.23,263.63) .. controls (20.23,257.91) and (23.47,253.28) .. (27.47,253.28) .. controls (31.47,253.28) and (34.71,257.91) .. (34.71,263.63) .. controls (34.71,269.35) and (31.47,273.99) .. (27.47,273.99) .. controls (23.47,273.99) and (20.23,269.35) .. (20.23,263.63) -- cycle ;
\draw    (27.07,283.48) -- (27.47,273.99) ;
\draw    (33.1,255.58) -- (37.13,248.96) ;
\draw    (20.44,241.48) -- (24.98,253.5) ;
\draw   (62.4,293.84) .. controls (62.4,288.12) and (65.64,283.48) .. (69.64,283.48) .. controls (73.64,283.48) and (76.89,288.12) .. (76.89,293.84) .. controls (76.89,299.56) and (73.64,304.2) .. (69.64,304.2) .. controls (65.64,304.2) and (62.4,299.56) .. (62.4,293.84) -- cycle ;
\draw   (63.86,236.02) .. controls (63.86,230.3) and (66.79,225.66) .. (70.41,225.66) .. controls (74.02,225.66) and (76.96,230.3) .. (76.96,236.02) .. controls (76.96,241.74) and (74.02,246.37) .. (70.41,246.37) .. controls (66.79,246.37) and (63.86,241.74) .. (63.86,236.02) -- cycle ;
\draw   (62.8,263.63) .. controls (62.8,257.91) and (66.05,253.28) .. (70.05,253.28) .. controls (74.05,253.28) and (77.29,257.91) .. (77.29,263.63) .. controls (77.29,269.35) and (74.05,273.99) .. (70.05,273.99) .. controls (66.05,273.99) and (62.8,269.35) .. (62.8,263.63) -- cycle ;
\draw    (69.64,283.48) -- (70.05,273.99) ;
\draw    (70.41,246.37) -- (70.05,253.28) ;
\draw   (104.61,295.57) .. controls (104.61,289.85) and (107.85,285.21) .. (111.86,285.21) .. controls (115.86,285.21) and (119.1,289.85) .. (119.1,295.57) .. controls (119.1,301.29) and (115.86,305.92) .. (111.86,305.92) .. controls (107.85,305.92) and (104.61,301.29) .. (104.61,295.57) -- cycle ;
\draw   (90.42,244.65) .. controls (90.42,238.93) and (93.35,234.29) .. (96.97,234.29) .. controls (100.59,234.29) and (103.52,238.93) .. (103.52,244.65) .. controls (103.52,250.37) and (100.59,255) .. (96.97,255) .. controls (93.35,255) and (90.42,250.37) .. (90.42,244.65) -- cycle ;
\draw   (106.69,236.59) .. controls (106.69,230.87) and (109.93,226.23) .. (113.93,226.23) .. controls (117.93,226.23) and (121.18,230.87) .. (121.18,236.59) .. controls (121.18,242.31) and (117.93,246.95) .. (113.93,246.95) .. controls (109.93,246.95) and (106.69,242.31) .. (106.69,236.59) -- cycle ;
\draw   (105.01,265.36) .. controls (105.01,259.64) and (108.26,255) .. (112.26,255) .. controls (116.26,255) and (119.5,259.64) .. (119.5,265.36) .. controls (119.5,271.08) and (116.26,275.72) .. (112.26,275.72) .. controls (108.26,275.72) and (105.01,271.08) .. (105.01,265.36) -- cycle ;
\draw    (111.86,285.21) -- (112.26,275.72) ;
\draw    (112.26,255) -- (111.89,247.24) ;
\draw    (101.34,252.13) -- (104.98,263.35) ;
\draw   (125.72,212.42) .. controls (125.72,206.7) and (128.65,202.07) .. (132.27,202.07) .. controls (135.88,202.07) and (138.82,206.7) .. (138.82,212.42) .. controls (138.82,218.14) and (135.88,222.78) .. (132.27,222.78) .. controls (128.65,222.78) and (125.72,218.14) .. (125.72,212.42) -- cycle ;
\draw   (124.7,241.19) .. controls (124.7,235.47) and (127.94,230.84) .. (131.94,230.84) .. controls (135.94,230.84) and (139.19,235.47) .. (139.19,241.19) .. controls (139.19,246.91) and (135.94,251.55) .. (131.94,251.55) .. controls (127.94,251.55) and (124.7,246.91) .. (124.7,241.19) -- cycle ;
\draw    (132.27,222.78) -- (131.91,229.69) ;
\draw    (129.72,249.82) -- (119.53,261.62) ;
\draw   (163.56,296.14) .. controls (163.56,290.42) and (166.8,285.79) .. (170.81,285.79) .. controls (174.81,285.79) and (178.05,290.42) .. (178.05,296.14) .. controls (178.05,301.86) and (174.81,306.5) .. (170.81,306.5) .. controls (166.8,306.5) and (163.56,301.86) .. (163.56,296.14) -- cycle ;
\draw   (143.26,221.06) .. controls (143.26,215.34) and (146.19,210.7) .. (149.81,210.7) .. controls (153.42,210.7) and (156.36,215.34) .. (156.36,221.06) .. controls (156.36,226.78) and (153.42,231.41) .. (149.81,231.41) .. controls (146.19,231.41) and (143.26,226.78) .. (143.26,221.06) -- cycle ;
\draw   (156.18,244.07) .. controls (156.18,238.35) and (159.42,233.71) .. (163.42,233.71) .. controls (167.42,233.71) and (170.66,238.35) .. (170.66,244.07) .. controls (170.66,249.79) and (167.42,254.43) .. (163.42,254.43) .. controls (159.42,254.43) and (156.18,249.79) .. (156.18,244.07) -- cycle ;
\draw   (163.96,265.94) .. controls (163.96,260.22) and (167.21,255.58) .. (171.21,255.58) .. controls (175.21,255.58) and (178.45,260.22) .. (178.45,265.94) .. controls (178.45,271.66) and (175.21,276.29) .. (171.21,276.29) .. controls (167.21,276.29) and (163.96,271.66) .. (163.96,265.94) -- cycle ;
\draw    (170.81,285.79) -- (171.21,276.29) ;
\draw    (167.22,258.17) -- (165.73,252.41) ;
\draw    (154.61,229.4) -- (157.88,236.88) ;
\draw   (160.29,218.18) .. controls (160.29,212.46) and (163.22,207.82) .. (166.84,207.82) .. controls (170.45,207.82) and (173.39,212.46) .. (173.39,218.18) .. controls (173.39,223.9) and (170.45,228.54) .. (166.84,228.54) .. controls (163.22,228.54) and (160.29,223.9) .. (160.29,218.18) -- cycle ;
\draw   (183.65,241.77) .. controls (183.65,236.05) and (186.89,231.41) .. (190.89,231.41) .. controls (194.89,231.41) and (198.14,236.05) .. (198.14,241.77) .. controls (198.14,247.49) and (194.89,252.13) .. (190.89,252.13) .. controls (186.89,252.13) and (183.65,247.49) .. (183.65,241.77) -- cycle ;
\draw    (166.84,228.54) -- (166.48,235.44) ;
\draw    (188.67,250.4) -- (178.48,262.2) ;
\draw   (219.74,296.14) .. controls (219.74,290.42) and (222.98,285.79) .. (226.98,285.79) .. controls (230.99,285.79) and (234.23,290.42) .. (234.23,296.14) .. controls (234.23,301.86) and (230.99,306.5) .. (226.98,306.5) .. controls (222.98,306.5) and (219.74,301.86) .. (219.74,296.14) -- cycle ;
\draw   (213.81,188.26) .. controls (213.81,182.54) and (216.74,177.9) .. (220.36,177.9) .. controls (223.97,177.9) and (226.91,182.54) .. (226.91,188.26) .. controls (226.91,193.98) and (223.97,198.62) .. (220.36,198.62) .. controls (216.74,198.62) and (213.81,193.98) .. (213.81,188.26) -- cycle ;
\draw   (212.36,244.07) .. controls (212.36,238.35) and (215.6,233.71) .. (219.6,233.71) .. controls (223.6,233.71) and (226.84,238.35) .. (226.84,244.07) .. controls (226.84,249.79) and (223.6,254.43) .. (219.6,254.43) .. controls (215.6,254.43) and (212.36,249.79) .. (212.36,244.07) -- cycle ;
\draw   (220.14,265.94) .. controls (220.14,260.22) and (223.39,255.58) .. (227.39,255.58) .. controls (231.39,255.58) and (234.63,260.22) .. (234.63,265.94) .. controls (234.63,271.66) and (231.39,276.29) .. (227.39,276.29) .. controls (223.39,276.29) and (220.14,271.66) .. (220.14,265.94) -- cycle ;
\draw    (226.98,285.79) -- (227.39,276.29) ;
\draw    (223.39,258.17) -- (221.91,252.41) ;
\draw    (218.17,198.9) -- (218.23,206.1) ;
\draw   (211.68,216.45) .. controls (211.68,210.73) and (214.61,206.1) .. (218.23,206.1) .. controls (221.84,206.1) and (224.78,210.73) .. (224.78,216.45) .. controls (224.78,222.17) and (221.84,226.81) .. (218.23,226.81) .. controls (214.61,226.81) and (211.68,222.17) .. (211.68,216.45) -- cycle ;
\draw   (239.83,241.77) .. controls (239.83,236.05) and (243.07,231.41) .. (247.07,231.41) .. controls (251.07,231.41) and (254.32,236.05) .. (254.32,241.77) .. controls (254.32,247.49) and (251.07,252.13) .. (247.07,252.13) .. controls (243.07,252.13) and (239.83,247.49) .. (239.83,241.77) -- cycle ;
\draw    (218.23,226.81) -- (217.87,233.71) ;
\draw    (244.85,250.4) -- (234.66,262.2) ;
\draw   (264.26,292.04) .. controls (264.26,286.32) and (267.5,281.69) .. (271.5,281.69) .. controls (275.5,281.69) and (278.75,286.32) .. (278.75,292.04) .. controls (278.75,297.76) and (275.5,302.4) .. (271.5,302.4) .. controls (267.5,302.4) and (264.26,297.76) .. (264.26,292.04) -- cycle ;
\draw   (265.71,234.22) .. controls (265.71,228.5) and (268.65,223.86) .. (272.26,223.86) .. controls (275.88,223.86) and (278.81,228.5) .. (278.81,234.22) .. controls (278.81,239.94) and (275.88,244.57) .. (272.26,244.57) .. controls (268.65,244.57) and (265.71,239.94) .. (265.71,234.22) -- cycle ;
\draw   (264.66,261.84) .. controls (264.66,256.12) and (267.9,251.48) .. (271.9,251.48) .. controls (275.9,251.48) and (279.15,256.12) .. (279.15,261.84) .. controls (279.15,267.56) and (275.9,272.19) .. (271.9,272.19) .. controls (267.9,272.19) and (264.66,267.56) .. (264.66,261.84) -- cycle ;
\draw    (271.5,281.69) -- (271.9,272.19) ;
\draw    (272.26,244.57) -- (271.9,251.48) ;
\draw   (283.03,204.9) .. controls (283.03,199.18) and (286.27,194.54) .. (290.27,194.54) .. controls (294.27,194.54) and (297.52,199.18) .. (297.52,204.9) .. controls (297.52,210.62) and (294.27,215.26) .. (290.27,215.26) .. controls (286.27,215.26) and (283.03,210.62) .. (283.03,204.9) -- cycle ;
\draw    (288.05,213.53) -- (277.86,225.32) ;
\draw   (258.54,206.3) .. controls (258.54,200.58) and (261.78,195.94) .. (265.78,195.94) .. controls (269.78,195.94) and (273.03,200.58) .. (273.03,206.3) .. controls (273.03,212.02) and (269.78,216.66) .. (265.78,216.66) .. controls (261.78,216.66) and (258.54,212.02) .. (258.54,206.3) -- cycle ;
\draw   (257.86,178.68) .. controls (257.86,172.96) and (260.79,168.33) .. (264.41,168.33) .. controls (268.02,168.33) and (270.96,172.96) .. (270.96,178.68) .. controls (270.96,184.4) and (268.02,189.04) .. (264.41,189.04) .. controls (260.79,189.04) and (257.86,184.4) .. (257.86,178.68) -- cycle ;
\draw    (264.41,189.04) -- (264.05,195.94) ;
\draw    (265.78,216.66) -- (269.06,224.14) ;
\draw   (322.96,293.45) .. controls (322.96,287.73) and (326.21,283.09) .. (330.21,283.09) .. controls (334.21,283.09) and (337.45,287.73) .. (337.45,293.45) .. controls (337.45,299.17) and (334.21,303.8) .. (330.21,303.8) .. controls (326.21,303.8) and (322.96,299.17) .. (322.96,293.45) -- cycle ;
\draw   (302.66,218.36) .. controls (302.66,212.64) and (305.59,208) .. (309.21,208) .. controls (312.82,208) and (315.76,212.64) .. (315.76,218.36) .. controls (315.76,224.08) and (312.82,228.71) .. (309.21,228.71) .. controls (305.59,228.71) and (302.66,224.08) .. (302.66,218.36) -- cycle ;
\draw   (315.58,241.37) .. controls (315.58,235.65) and (318.82,231.02) .. (322.82,231.02) .. controls (326.82,231.02) and (330.06,235.65) .. (330.06,241.37) .. controls (330.06,247.09) and (326.82,251.73) .. (322.82,251.73) .. controls (318.82,251.73) and (315.58,247.09) .. (315.58,241.37) -- cycle ;
\draw   (323.36,263.24) .. controls (323.36,257.52) and (326.61,252.88) .. (330.61,252.88) .. controls (334.61,252.88) and (337.85,257.52) .. (337.85,263.24) .. controls (337.85,268.96) and (334.61,273.59) .. (330.61,273.59) .. controls (326.61,273.59) and (323.36,268.96) .. (323.36,263.24) -- cycle ;
\draw    (330.21,283.09) -- (330.61,273.59) ;
\draw    (326.62,255.47) -- (325.13,249.72) ;
\draw    (314.01,226.7) -- (317.28,234.18) ;
\draw   (319.69,215.48) .. controls (319.69,209.76) and (322.62,205.12) .. (326.24,205.12) .. controls (329.86,205.12) and (332.79,209.76) .. (332.79,215.48) .. controls (332.79,221.2) and (329.86,225.84) .. (326.24,225.84) .. controls (322.62,225.84) and (319.69,221.2) .. (319.69,215.48) -- cycle ;
\draw   (343.05,239.07) .. controls (343.05,233.35) and (346.29,228.71) .. (350.29,228.71) .. controls (354.29,228.71) and (357.54,233.35) .. (357.54,239.07) .. controls (357.54,244.79) and (354.29,249.43) .. (350.29,249.43) .. controls (346.29,249.43) and (343.05,244.79) .. (343.05,239.07) -- cycle ;
\draw    (326.24,225.84) -- (325.88,232.74) ;
\draw    (348.07,247.7) -- (337.88,259.5) ;
\draw   (365.56,293.84) .. controls (365.56,288.12) and (368.8,283.48) .. (372.8,283.48) .. controls (376.8,283.48) and (380.05,288.12) .. (380.05,293.84) .. controls (380.05,299.56) and (376.8,304.2) .. (372.8,304.2) .. controls (368.8,304.2) and (365.56,299.56) .. (365.56,293.84) -- cycle ;
\draw   (367.02,236.02) .. controls (367.02,230.3) and (369.95,225.66) .. (373.57,225.66) .. controls (377.18,225.66) and (380.12,230.3) .. (380.12,236.02) .. controls (380.12,241.74) and (377.18,246.37) .. (373.57,246.37) .. controls (369.95,246.37) and (367.02,241.74) .. (367.02,236.02) -- cycle ;
\draw   (365.96,263.63) .. controls (365.96,257.91) and (369.21,253.28) .. (373.21,253.28) .. controls (377.21,253.28) and (380.45,257.91) .. (380.45,263.63) .. controls (380.45,269.35) and (377.21,273.99) .. (373.21,273.99) .. controls (369.21,273.99) and (365.96,269.35) .. (365.96,263.63) -- cycle ;
\draw    (372.8,283.48) -- (373.21,273.99) ;
\draw    (373.57,246.37) -- (373.21,253.28) ;
\draw   (384.33,206.7) .. controls (384.33,200.98) and (387.57,196.34) .. (391.58,196.34) .. controls (395.58,196.34) and (398.82,200.98) .. (398.82,206.7) .. controls (398.82,212.42) and (395.58,217.05) .. (391.58,217.05) .. controls (387.57,217.05) and (384.33,212.42) .. (384.33,206.7) -- cycle ;
\draw    (389.35,215.33) -- (379.16,227.12) ;
\draw   (359.84,208.1) .. controls (359.84,202.38) and (363.08,197.74) .. (367.08,197.74) .. controls (371.08,197.74) and (374.33,202.38) .. (374.33,208.1) .. controls (374.33,213.82) and (371.08,218.46) .. (367.08,218.46) .. controls (363.08,218.46) and (359.84,213.82) .. (359.84,208.1) -- cycle ;
\draw   (359.16,180.48) .. controls (359.16,174.76) and (362.09,170.12) .. (365.71,170.12) .. controls (369.33,170.12) and (372.26,174.76) .. (372.26,180.48) .. controls (372.26,186.2) and (369.33,190.84) .. (365.71,190.84) .. controls (362.09,190.84) and (359.16,186.2) .. (359.16,180.48) -- cycle ;
\draw    (365.71,190.84) -- (365.35,197.74) ;
\draw    (367.08,218.46) -- (370.36,225.94) ;
\draw   (410.66,294.28) .. controls (410.66,288.56) and (413.9,283.92) .. (417.9,283.92) .. controls (421.9,283.92) and (425.15,288.56) .. (425.15,294.28) .. controls (425.15,300) and (421.9,304.63) .. (417.9,304.63) .. controls (413.9,304.63) and (410.66,300) .. (410.66,294.28) -- cycle ;
\draw   (404.72,186.39) .. controls (404.72,180.67) and (407.66,176.03) .. (411.27,176.03) .. controls (414.89,176.03) and (417.82,180.67) .. (417.82,186.39) .. controls (417.82,192.11) and (414.89,196.75) .. (411.27,196.75) .. controls (407.66,196.75) and (404.72,192.11) .. (404.72,186.39) -- cycle ;
\draw   (403.27,242.2) .. controls (403.27,236.48) and (406.52,231.85) .. (410.52,231.85) .. controls (414.52,231.85) and (417.76,236.48) .. (417.76,242.2) .. controls (417.76,247.92) and (414.52,252.56) .. (410.52,252.56) .. controls (406.52,252.56) and (403.27,247.92) .. (403.27,242.2) -- cycle ;
\draw   (411.06,264.07) .. controls (411.06,258.35) and (414.3,253.71) .. (418.3,253.71) .. controls (422.3,253.71) and (425.55,258.35) .. (425.55,264.07) .. controls (425.55,269.79) and (422.3,274.43) .. (418.3,274.43) .. controls (414.3,274.43) and (411.06,269.79) .. (411.06,264.07) -- cycle ;
\draw    (417.9,283.92) -- (418.3,274.43) ;
\draw    (414.31,256.3) -- (412.83,250.55) ;
\draw    (409.09,197.04) -- (409.14,204.23) ;
\draw   (402.59,214.59) .. controls (402.59,208.87) and (405.53,204.23) .. (409.14,204.23) .. controls (412.76,204.23) and (415.69,208.87) .. (415.69,214.59) .. controls (415.69,220.31) and (412.76,224.94) .. (409.14,224.94) .. controls (405.53,224.94) and (402.59,220.31) .. (402.59,214.59) -- cycle ;
\draw   (430.74,239.9) .. controls (430.74,234.18) and (433.99,229.55) .. (437.99,229.55) .. controls (441.99,229.55) and (445.23,234.18) .. (445.23,239.9) .. controls (445.23,245.62) and (441.99,250.26) .. (437.99,250.26) .. controls (433.99,250.26) and (430.74,245.62) .. (430.74,239.9) -- cycle ;
\draw    (409.14,224.94) -- (408.78,231.85) ;
\draw    (435.77,248.53) -- (425.58,260.33) ;
\draw   (451.22,292.91) .. controls (451.22,287.19) and (454.46,282.55) .. (458.46,282.55) .. controls (462.46,282.55) and (465.71,287.19) .. (465.71,292.91) .. controls (465.71,298.63) and (462.46,303.26) .. (458.46,303.26) .. controls (454.46,303.26) and (451.22,298.63) .. (451.22,292.91) -- cycle ;
\draw   (452.67,235.08) .. controls (452.67,229.36) and (455.61,224.72) .. (459.22,224.72) .. controls (462.84,224.72) and (465.77,229.36) .. (465.77,235.08) .. controls (465.77,240.8) and (462.84,245.44) .. (459.22,245.44) .. controls (455.61,245.44) and (452.67,240.8) .. (452.67,235.08) -- cycle ;
\draw   (451.62,262.7) .. controls (451.62,256.98) and (454.86,252.34) .. (458.86,252.34) .. controls (462.87,252.34) and (466.11,256.98) .. (466.11,262.7) .. controls (466.11,268.42) and (462.87,273.06) .. (458.86,273.06) .. controls (454.86,273.06) and (451.62,268.42) .. (451.62,262.7) -- cycle ;
\draw    (458.46,282.55) -- (458.86,273.06) ;
\draw    (459.22,245.44) -- (458.86,252.34) ;
\draw   (436.55,214.77) .. controls (436.55,209.05) and (439.49,204.41) .. (443.1,204.41) .. controls (446.72,204.41) and (449.65,209.05) .. (449.65,214.77) .. controls (449.65,220.49) and (446.72,225.13) .. (443.1,225.13) .. controls (439.49,225.13) and (436.55,220.49) .. (436.55,214.77) -- cycle ;
\draw   (452.82,206.72) .. controls (452.82,201) and (456.06,196.36) .. (460.07,196.36) .. controls (464.07,196.36) and (467.31,201) .. (467.31,206.72) .. controls (467.31,212.44) and (464.07,217.07) .. (460.07,217.07) .. controls (456.06,217.07) and (452.82,212.44) .. (452.82,206.72) -- cycle ;
\draw    (458.39,225.13) -- (458.02,217.36) ;
\draw    (447.47,222.25) -- (451.11,233.47) ;
\draw   (470.83,211.32) .. controls (470.83,205.6) and (474.08,200.96) .. (478.08,200.96) .. controls (482.08,200.96) and (485.32,205.6) .. (485.32,211.32) .. controls (485.32,217.04) and (482.08,221.68) .. (478.08,221.68) .. controls (474.08,221.68) and (470.83,217.04) .. (470.83,211.32) -- cycle ;
\draw    (475.85,219.95) -- (465.67,231.74) ;
\draw   (492.1,291.04) .. controls (492.1,285.32) and (495.34,280.68) .. (499.34,280.68) .. controls (503.34,280.68) and (506.59,285.32) .. (506.59,291.04) .. controls (506.59,296.76) and (503.34,301.4) .. (499.34,301.4) .. controls (495.34,301.4) and (492.1,296.76) .. (492.1,291.04) -- cycle ;
\draw   (493.55,233.21) .. controls (493.55,227.49) and (496.49,222.86) .. (500.1,222.86) .. controls (503.72,222.86) and (506.65,227.49) .. (506.65,233.21) .. controls (506.65,238.93) and (503.72,243.57) .. (500.1,243.57) .. controls (496.49,243.57) and (493.55,238.93) .. (493.55,233.21) -- cycle ;
\draw   (492.5,260.83) .. controls (492.5,255.11) and (495.74,250.48) .. (499.74,250.48) .. controls (503.74,250.48) and (506.99,255.11) .. (506.99,260.83) .. controls (506.99,266.55) and (503.74,271.19) .. (499.74,271.19) .. controls (495.74,271.19) and (492.5,266.55) .. (492.5,260.83) -- cycle ;
\draw    (499.34,280.68) -- (499.74,271.19) ;
\draw    (500.1,243.57) -- (499.74,250.48) ;
\draw   (493.7,204.85) .. controls (493.7,199.13) and (496.94,194.49) .. (500.94,194.49) .. controls (504.94,194.49) and (508.19,199.13) .. (508.19,204.85) .. controls (508.19,210.57) and (504.94,215.2) .. (500.94,215.2) .. controls (496.94,215.2) and (493.7,210.57) .. (493.7,204.85) -- cycle ;
\draw    (499.27,223.26) -- (498.9,215.49) ;
\draw   (512.51,176.77) .. controls (512.51,171.05) and (515.76,166.42) .. (519.76,166.42) .. controls (523.76,166.42) and (527,171.05) .. (527,176.77) .. controls (527,182.49) and (523.76,187.13) .. (519.76,187.13) .. controls (515.76,187.13) and (512.51,182.49) .. (512.51,176.77) -- cycle ;
\draw    (517.53,185.41) -- (507.35,197.2) ;
\draw   (480.64,177.43) .. controls (480.64,171.71) and (483.57,167.07) .. (487.19,167.07) .. controls (490.81,167.07) and (493.74,171.71) .. (493.74,177.43) .. controls (493.74,183.15) and (490.81,187.78) .. (487.19,187.78) .. controls (483.57,187.78) and (480.64,183.15) .. (480.64,177.43) -- cycle ;
\draw    (491.56,184.91) -- (495.19,196.13) ;

\draw (20.67,290) node [anchor=north west][inner sep=0.75pt]    {$1$};
\draw (21.87,255.75) node [anchor=north west][inner sep=0.75pt]    {$2$};
\draw (12,230) node [anchor=north west][inner sep=0.75pt]    {$3$};
\draw (35.56,228.99) node [anchor=north west][inner sep=0.75pt]    {$4$};
\draw (37.5,279.2) node [anchor=north west][inner sep=0.75pt]    {$\circ _{2}$};
\draw (64,288) node [anchor=north west][inner sep=0.75pt]    {$a$};
\draw (64,256) node [anchor=north west][inner sep=0.75pt]    {$b$};
\draw (64,230) node [anchor=north west][inner sep=0.75pt]    {$c$};
\draw (78.78,280.78) node [anchor=north west][inner sep=0.75pt]    {$=$};
\draw (105.45,290) node [anchor=north west][inner sep=0.75pt]    {$1$};
\draw (90.24,235.32) node [anchor=north west][inner sep=0.75pt]    {$3$};
\draw (106.51,230.14) node [anchor=north west][inner sep=0.75pt]    {$4$};
\draw (126.31,232.16) node [anchor=north west][inner sep=0.75pt]    {$b$};
\draw (127,206) node [anchor=north west][inner sep=0.75pt]    {$c$};
\draw (164.4,290) node [anchor=north west][inner sep=0.75pt]    {$1$};
\draw (143.81,212.3) node [anchor=north west][inner sep=0.75pt]    {$c$};
\draw (156,237.62) node [anchor=north west][inner sep=0.75pt]    {$b$};
\draw (185.26,232.73) node [anchor=north west][inner sep=0.75pt]    {$4$};
\draw (161.93,210) node [anchor=north west][inner sep=0.75pt]    {$3$};
\draw (105.82,259) node [anchor=north west][inner sep=0.75pt]    {$a$};
\draw (165.13,259) node [anchor=north west][inner sep=0.75pt]    {$a$};
\draw (131.41,280.08) node [anchor=north west][inner sep=0.75pt]    {$+$};
\draw (220.58,290) node [anchor=north west][inner sep=0.75pt]    {$1$};
\draw (214.36,179.51) node [anchor=north west][inner sep=0.75pt]    {$3$};
\draw (212.18,237.62) node [anchor=north west][inner sep=0.75pt]    {$b$};
\draw (241.44,232.73) node [anchor=north west][inner sep=0.75pt]    {$4$};
\draw (213.32,211) node [anchor=north west][inner sep=0.75pt]    {$c$};
\draw (221.31,259) node [anchor=north west][inner sep=0.75pt]    {$a$};
\draw (188.89,280.08) node [anchor=north west][inner sep=0.75pt]    {$+$};
\draw (265.1,290) node [anchor=north west][inner sep=0.75pt]    {$1$};
\draw (266.31,259) node [anchor=north west][inner sep=0.75pt]    {$a$};
\draw (266.99,225.47) node [anchor=north west][inner sep=0.75pt]    {$b$};
\draw (284.64,195.86) node [anchor=north west][inner sep=0.75pt]    {$4$};
\draw (260.28,198) node [anchor=north west][inner sep=0.75pt]    {$c$};
\draw (259.5,170.51) node [anchor=north west][inner sep=0.75pt]    {$3$};
\draw (241.78,280.08) node [anchor=north west][inner sep=0.75pt]    {$+$};
\draw (323.8,287) node [anchor=north west][inner sep=0.75pt]    {$1$};
\draw (303.21,212) node [anchor=north west][inner sep=0.75pt]    {$c$};
\draw (315.4,234.92) node [anchor=north west][inner sep=0.75pt]    {$b$};
\draw (344.66,230.03) node [anchor=north west][inner sep=0.75pt]    {$3$};
\draw (321.33,207.31) node [anchor=north west][inner sep=0.75pt]    {$4$};
\draw (324.53,259) node [anchor=north west][inner sep=0.75pt]    {$a$};
\draw (297.64,280.08) node [anchor=north west][inner sep=0.75pt]    {$+$};
\draw (366.4,287) node [anchor=north west][inner sep=0.75pt]    {$1$};
\draw (367.61,259) node [anchor=north west][inner sep=0.75pt]    {$a$};
\draw (368.29,227.26) node [anchor=north west][inner sep=0.75pt]    {$b$};
\draw (385.94,197.66) node [anchor=north west][inner sep=0.75pt]    {$3$};
\draw (361.58,200) node [anchor=north west][inner sep=0.75pt]    {$c$};
\draw (360.8,172.31) node [anchor=north west][inner sep=0.75pt]    {$4$};
\draw (344.57,280.08) node [anchor=north west][inner sep=0.75pt]    {$+$};
\draw (411.5,290) node [anchor=north west][inner sep=0.75pt]    {$1$};
\draw (405.27,177.64) node [anchor=north west][inner sep=0.75pt]    {$4$};
\draw (403.1,235.75) node [anchor=north west][inner sep=0.75pt]    {$b$};
\draw (404.23,209) node [anchor=north west][inner sep=0.75pt]    {$c$};
\draw (412.23,259) node [anchor=north west][inner sep=0.75pt]    {$a$};
\draw (387.25,280.08) node [anchor=north west][inner sep=0.75pt]    {$+$};
\draw (431.73,230.33) node [anchor=north west][inner sep=0.75pt]    {$3$};
\draw (452.06,290) node [anchor=north west][inner sep=0.75pt]    {$1$};
\draw (453.27,259) node [anchor=north west][inner sep=0.75pt]    {$a$};
\draw (453.95,226.33) node [anchor=north west][inner sep=0.75pt]    {$b$};
\draw (436.38,205.44) node [anchor=north west][inner sep=0.75pt]    {$3$};
\draw (455.75,200) node [anchor=north west][inner sep=0.75pt]    {$c$};
\draw (472.44,202.28) node [anchor=north west][inner sep=0.75pt]    {$4$};
\draw (432.27,280.08) node [anchor=north west][inner sep=0.75pt]    {$+$};
\draw (492.94,290) node [anchor=north west][inner sep=0.75pt]    {$1$};
\draw (494.15,259) node [anchor=north west][inner sep=0.75pt]    {$a$};
\draw (494.83,224.46) node [anchor=north west][inner sep=0.75pt]    {$b$};
\draw (496.63,198) node [anchor=north west][inner sep=0.75pt]    {$c$};
\draw (514.12,167.74) node [anchor=north west][inner sep=0.75pt]    {$4$};
\draw (480.46,168.1) node [anchor=north west][inner sep=0.75pt]    {$3$};
\draw (469.94,280.08) node [anchor=north west][inner sep=0.75pt]    {$+$};

\end{tikzpicture}

\end{center}
  \section{Operad structure on $\mop{NAP} \circ \tq$}
\noindent From the substitution formula
	$$\mop{NAP} \circ \tq [I]=\displaystyle \bigoplus_{\pi \vdash I} \mop{NAP} [\pi] \otimes \tq(\pi),$$
	one can see that a typical element in $\mop{NAP} \circ \tq [I]$ is of the form $t \otimes \displaystyle \bigotimes_{C \in \pi}\gamma_C$, where $\pi \vdash I$ and $t \in \mop{NAP}[\pi]$, and where $\gamma_C\in \tq[C]$ for any block $C$ of $\pi$. A natural representation of this element is obtained by drawing the rooted tree $t$ and writing $\gamma_C$ in a bubble placed at vertex $C$ for any $C\in\pi$. We define the partial composition 
 \[\square_s:\mop{NAP}\circ \tq[S]\otimes \mop{NAP}\circ \tq[T]\to \mop{NAP}\circ \tq[S\sqcup_s T]\]
 as follows: let $\pi \vdash S$, let $ \rho \vdash T$, let $s \in S$ and let $C_s$ be the block of $\pi$ which contains $s$. Then
	$$\Bigg( t \otimes \displaystyle \bigotimes_{C \in \pi} \gamma_{C} \Bigg) \square_{s} \Bigg( u \otimes \displaystyle \bigotimes_{B \in \rho} \beta_{B} \Bigg):=(t\circ_{C_{s}}u) \otimes \bigotimes_{D \in \pi \sqcup_{s}\rho}\alpha_{D}$$
	with   
 \[
\left \{
\begin{array}{c  c}
    \alpha_{D}=\gamma_{C}  & \text{ if } D \in \pi\setminus \{C_{s}\}\\
    \alpha_{D}=\beta_{B}& \text{ if } D \in \rho\setminus \{B_{\smop{root}(u)}\}\\
    \alpha_{D}=\gamma_{C_{s}}\circ_{s} \beta_{B_{\smop{root}(u)}} &\text{ if } D= C_s \sqcup_s B_{\smop{root}(u)},\\& \text{ where }B_{\smop{root}(u)}\text{ is the block of } \rho\text{ in the root of the tree }u.
\end{array}
\right.
\]
\vskip 5mm
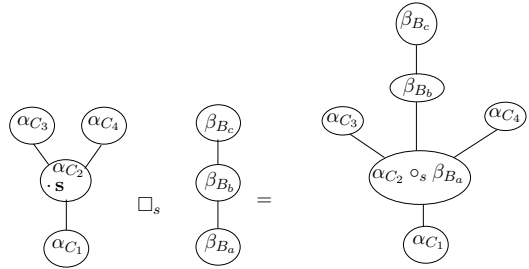
\begin{example}
\[
\scalebox{0.7}{
\begin{tikzpicture}[x=0.75pt,y=0.75pt,yscale=-1,xscale=1]

\draw   (25.63,341.68) .. controls (25.63,334.33) and (32.85,328.36) .. (41.77,328.36) .. controls (50.68,328.36) and (57.9,334.33) .. (57.9,341.68) .. controls (57.9,349.04) and (50.68,355) .. (41.77,355) .. controls (32.85,355) and (25.63,349.04) .. (25.63,341.68) -- cycle ;
\draw    (41.77,306.28) -- (41.77,328.36) ;
\draw   (23.08,292.61) .. controls (23.08,284.29) and (31.26,277.54) .. (41.34,277.54) .. controls (51.43,277.54) and (59.6,284.29) .. (59.6,292.61) .. controls (59.6,300.94) and (51.43,307.68) .. (41.34,307.68) .. controls (31.26,307.68) and (23.08,300.94) .. (23.08,292.61) -- cycle ;
\draw    (55,284) -- (67,267) ;
\draw    (18,266) -- (29,281) ;
\draw   (1,253.01) .. controls (1,245.65) and (8.22,239.69) .. (17.14,239.69) .. controls (26.05,239.69) and (33.27,245.65) .. (33.27,253.01) .. controls (33.27,260.36) and (26.05,266.33) .. (17.14,266.33) .. controls (8.22,266.33) and (1,260.36) .. (1,253.01) -- cycle ;
\draw   (52.81,253.01) .. controls (52.81,245.65) and (60.03,239.69) .. (68.95,239.69) .. controls (77.86,239.69) and (85.08,245.65) .. (85.08,253.01) .. controls (85.08,260.36) and (77.86,266.33) .. (68.95,266.33) .. controls (60.03,266.33) and (52.81,260.36) .. (52.81,253.01) -- cycle ;
\draw   (135.19,340.45) .. controls (135.19,333.2) and (142.32,327.31) .. (151.12,327.31) .. controls (159.91,327.31) and (167.04,333.2) .. (167.04,340.45) .. controls (167.04,347.71) and (159.91,353.6) .. (151.12,353.6) .. controls (142.32,353.6) and (135.19,347.71) .. (135.19,340.45) -- cycle ;
\draw   (135.4,251.08) .. controls (135.4,243.82) and (142.53,237.94) .. (151.33,237.94) .. controls (160.12,237.94) and (167.25,243.82) .. (167.25,251.08) .. controls (167.25,258.34) and (160.12,264.22) .. (151.33,264.22) .. controls (142.53,264.22) and (135.4,258.34) .. (135.4,251.08) -- cycle ;
\draw   (135.19,294.19) .. controls (135.19,286.93) and (142.32,281.05) .. (151.12,281.05) .. controls (159.91,281.05) and (167.04,286.93) .. (167.04,294.19) .. controls (167.04,301.45) and (159.91,307.33) .. (151.12,307.33) .. controls (142.32,307.33) and (135.19,301.45) .. (135.19,294.19) -- cycle ;
\draw    (151,263) -- (151,281) ;
\draw    (151,307) -- (151,327) ;
\draw   (284.67,338.88) .. controls (284.67,331.52) and (291.9,325.56) .. (300.81,325.56) .. controls (309.72,325.56) and (316.95,331.52) .. (316.95,338.88) .. controls (316.95,346.23) and (309.72,352.2) .. (300.81,352.2) .. controls (291.9,352.2) and (284.67,346.23) .. (284.67,338.88) -- cycle ;
\draw    (299,310) -- (299,326) ;
\draw   (260,290.42) .. controls (260,279.73) and (276.34,271.06) .. (296.5,271.06) .. controls (316.66,271.06) and (333,279.73) .. (333,290.42) .. controls (333,301.12) and (316.66,309.79) .. (296.5,309.79) .. controls (276.34,309.79) and (260,301.12) .. (260,290.42) -- cycle ;
\draw    (321,276) -- (352,256) ;
\draw    (246,259) -- (271,276) ;
\draw   (226.07,249.54) .. controls (226.07,243.9) and (232.72,239.32) .. (240.92,239.32) .. controls (249.12,239.32) and (255.77,243.9) .. (255.77,249.54) .. controls (255.77,255.18) and (249.12,259.76) .. (240.92,259.76) .. controls (232.72,259.76) and (226.07,255.18) .. (226.07,249.54) -- cycle ;
\draw   (343.3,246.31) .. controls (343.3,240.67) and (349.95,236.09) .. (358.15,236.09) .. controls (366.35,236.09) and (373,240.67) .. (373,246.31) .. controls (373,251.96) and (366.35,256.53) .. (358.15,256.53) .. controls (349.95,256.53) and (343.3,251.96) .. (343.3,246.31) -- cycle ;
\draw   (279.41,179.59) .. controls (279.41,171.25) and (285.97,164.5) .. (294.06,164.5) .. controls (302.16,164.5) and (308.72,171.25) .. (308.72,179.59) .. controls (308.72,187.92) and (302.16,194.67) .. (294.06,194.67) .. controls (285.97,194.67) and (279.41,187.92) .. (279.41,179.59) -- cycle ;
\draw   (275,225.74) .. controls (275,220.17) and (283.73,215.65) .. (294.5,215.65) .. controls (305.27,215.65) and (314,220.17) .. (314,225.74) .. controls (314,231.31) and (305.27,235.82) .. (294.5,235.82) .. controls (283.73,235.82) and (275,231.31) .. (275,225.74) -- cycle ;
\draw    (294,195) -- (294,216) ;
\draw    (294,236) -- (294,271) ;

\draw (30,334) node [anchor=north west][inner sep=0.75pt]    {$\alpha _{C_1}$};
\draw (30.54,279) node [anchor=north west][inner sep=0.75pt]    {$\alpha _{C_2}$};
\draw (26,292) node [anchor=north west][inner sep=0.75pt]    {$\cdot \, \mathbf s$};
\draw (56,245) node [anchor=north west][inner sep=0.75pt]    {$\alpha _{C_4}$};
\draw (5,245) node [anchor=north west][inner sep=0.75pt]    {$\alpha _{C_3}$};
\draw (91.77,303.64) node [anchor=north west][inner sep=0.75pt]    {$\square_{s}$};
\draw (140,330) node [anchor=north west][inner sep=0.75pt]    {$\beta _{B_a}$};
\draw (140,285) node [anchor=north west][inner sep=0.75pt]    {$\beta _{B_b}$};
\draw (140,243) node [anchor=north west][inner sep=0.75pt]    {$\beta _{B_c}$};
\draw (177.55,303.24) node [anchor=north west][inner sep=0.75pt]    {$=$};
\draw (290,331) node [anchor=north west][inner sep=0.75pt]    {$\alpha _{C_1}$};
\draw (260.79,279.21) node [anchor=north west][inner sep=0.75pt]   {$\alpha _{C_2} \circ _{s} \beta _{B_a}$};
\draw (228,243) node [anchor=north west][inner sep=0.75pt]    {$\alpha _{C_3}$};
\draw (345,239) node [anchor=north west][inner sep=0.75pt]    {$\alpha _{C_4}$};
\draw (283,217) node [anchor=north west][inner sep=0.75pt]    {$\beta _{B_b}$};
\draw (280.95,168.81) node [anchor=north west][inner sep=0.75pt]    {$\beta _{B_c}$};

\end{tikzpicture}
}
\]
\end{example}
	 \begin{theorem}
	 	$(\mop{NAP} \circ \tq, \square)$ is an operad.
	 \end{theorem}
	\begin{proof}\strut\\
        \begin{itemize}
            \item Associativity: let $I=S\sqcup T\sqcup R$,  let $\pi \vdash S$, $ \rho \vdash T$ and $\tau \vdash R$,  let $s,s' \in S$ and let $t \in T$. Let $C_{s}$, $C_{s'}$ be the blocks of $\pi$ which contain respectively $s$ and $s'$, and let 
           $A_{\smop{root}(v)}$ be the block of $\tau$ which contains the root of $v$.
            \begin{itemize}
		      \item Parallel associativity: two subcases must be considered.\\
        
		      \noindent - If $C_{s}=C_{s'}$, we have:
		      \begin{eqnarray*}
			& &\left(\bigg( t \otimes \displaystyle \bigotimes_{C \in \pi} \gamma_{C} \bigg) \square_{s} \bigg( u \otimes \displaystyle \bigotimes_{B \in \rho} \beta_{B} \bigg)\right) \square_{s'} \Bigg( v \otimes \bigotimes_{A\in \tau} \alpha_A \Bigg)\\
           &=&\left( (t\circ_{C_{s}}u) \otimes \Big(\bigotimes_{C \in  \pi\setminus \{C_{s}\}}\gamma_{c}\Big) \otimes (\gamma_{C_{s}} \circ_s \beta_{B_{\ssmop{root}(u)}}) \otimes \Big(\bigotimes_{B \in  \rho \setminus \{B_{\ssmop{root}(u)}\}}\beta_{B} \Big)\right)   \square_{s'} \Bigg( v \otimes \bigotimes_{A\in \tau} \alpha_A \Bigg) \\
			&=&\Big((t\circ_{C_{s}}u)\circ_{B_{\ssmop{root}(u)}}v \Big) \otimes \Big(\bigotimes_{C \in  \pi\setminus \{C_{s}\}}\gamma_{c}\Big) \otimes \Big((\gamma_{C_{s}} \circ_s \beta_{B_{\ssmop{root}(u)}})\circ_{s'} \alpha_{A_{\ssmop{root}(v)}}\Big)\otimes \\
   &&\hskip 90mm \Big(\bigotimes_{B \in  \rho \setminus\{B_{\ssmop{root}(u)}\}}\beta_{B}\Big) \otimes\Big(\bigotimes_{A\in \tau \setminus \{A_{\ssmop{root}(v)}\}}\alpha_A\Big)\\
   &&\hbox{(using the fact that $s'$ is an element of the new block $C_s\sqcup_s B_{\ssmop{root}(u)}$)}\\
			&=&\Big((t\circ_{C_{s}}v) \circ_{B_{\ssmop{root}(v)}} u \Big) \otimes \Big(\bigotimes_{C \in  \pi\setminus C_{s}}\gamma_{c}\Big) \otimes \Big((\gamma_{C_{s}} \circ_{s'} \alpha_{A_{\ssmop{root}(v)}} )\circ_s \beta_{B_{\ssmop{root}(u)}}\Big)\\ 
            & & \otimes \Big(\bigotimes_{B \in  \rho \setminus \{B_{\ssmop{root}(u)}\}}\beta_{B}\Big) \otimes \Big(\bigotimes_{A\in \tau \setminus \{A_{\ssmop{root}(v)}\}} \alpha_A\Big) \\
			& & \text{(via Property \eqref{asso-suppl} for $\mop{NAP }$ and parallel associativity of $\tq$)}\\
			&=&\left(\bigg( t \otimes \displaystyle \bigotimes_{C \in \pi} \gamma_{C} \bigg) \square_{s'} \bigg( v \otimes \bigotimes_{A\in \tau} \alpha_A \bigg)\right) \square_{s} \Bigg( u \otimes \displaystyle \bigotimes_{B \in \rho} \beta_{B} \Bigg).
		\end{eqnarray*}

  Here we should precise that $t\circ_{C_{s}}(v\circ_{B_{\ssmop{root}(u)}}u)$ labellizes the vertices of our new tree with a partition $\chi$ of $S\sqcup T\sqcup R\setminus\{s,s'\}$, whose blocks are the blocks of $\pi$ except $C_s$, the blocks of $\rho$ except $B_{\ssmop{root}(u)}$, the blocks of $\tau$ except $A_{\ssmop{root}(v)}$ and the block $(C_s \sqcup B_{\ssmop{root}(u)} \sqcup A_{\ssmop{root}(v)})\setminus{\{s,s'\}}. $
  
\smallbreak
  
		\noindent - If $ C_s \ne  C_{s'}$, we have:
			\begin{eqnarray*}
				& &\left(\bigg( t \otimes \displaystyle \bigotimes_{C \in \pi} \gamma_{C} \bigg) \square_{s} \bigg( u \otimes \displaystyle \bigotimes_{B \in \rho} \beta_{B} \bigg)\right) \square_{s'} \Bigg( v \otimes \bigotimes_{A\in \tau} \alpha_A \Bigg)\\
                &=& \left( (t\circ_{C_{s}}u) \otimes \Big(\bigotimes_{C \in  \pi\setminus \{C_{s}\}}\gamma_{C}\Big) \otimes (\gamma_{C_{s}} \circ_s \beta_{B_{\ssmop{root}(u)}}) \otimes \Big(\bigotimes_{C \in  \rho \setminus \{B_{\ssmop{root}(u)}\}}\beta_{B}\Big) \right) \square_{s'} \Bigg( v \otimes \bigotimes_{A\in \tau} \alpha_A \Bigg) \\
				&=&\Big((t\circ_{C_{s}}u)\circ_{C_{s'}}v \Big) \otimes \Big(\bigotimes_{C \in  \pi\setminus\{ C_{s}, C_{s'}\}}\gamma_{C}\Big) \otimes (\gamma_{C_{s}} \circ_s \beta_{B_{\ssmop{root}(u)}})\otimes (\gamma_{C_{s'}}\circ_{s'} \alpha_{A_{\ssmop{root}(v)}}\Big)\\
    &&\otimes \Big(\bigotimes_{C \in  \rho \setminus \{B_{\ssmop{root}(u)}\}}\beta_{B}\Big)
                \otimes \Big(\bigotimes_{A\in \tau \setminus \{A_{\ssmop{root}(v)}\}} \alpha_A\Big)\\
				&=&\Big((t\circ_{C_{s'}}v)\circ_{C_{s}}u \Big) \otimes \Big(\bigotimes_{C \in  \pi\setminus \{C_{s},C_{s'}\}}\gamma_{C}\Big) \otimes (\gamma_{C_{s'}} \circ_{s'} \alpha_{A_{\ssmop{root}(v)}} )\otimes (\gamma_{C_{s}}\circ_s \beta_{B_{\ssmop{root}(u)}})\\
    &&\otimes \Big(\bigotimes_{C \in  \rho \setminus \{B_{\ssmop{root}(u)}\}}\beta_{B}\Big) \otimes \Big(\bigotimes_{A\in \tau \setminus \{A_{\ssmop{root}(v)}\} }\alpha_A\Big) 
            \end{eqnarray*}
			via parallel associativity of $\mop{NAP}$. Hence,
            \begin{eqnarray*}
            & &\left(\bigg( t \otimes \displaystyle \bigotimes_{C \in \pi} \gamma_{C} \bigg) \square_{s} \bigg( u \otimes \displaystyle \bigotimes_{B \in \rho} \beta_{B} \bigg)\right) \square_{s'} \Bigg( v \otimes \bigotimes_{A\in \tau} \alpha_A \Bigg)\\
				&=&\left(\bigg( t \otimes \displaystyle \bigotimes_{C \in \pi} \gamma_{C} \bigg) \square_{s'} \bigg( v \otimes \bigotimes_{A\in \tau} \alpha_A \bigg)\right) \square_{s} \Bigg( u \otimes \displaystyle \bigotimes_{B \in \rho} \beta_{B} \Bigg).
			\end{eqnarray*}
		\item Nested associativity: Two subcases again arise.\\
  
		\noindent - If $b \in B_{\smop{root}(u)}$, we have:
		\begin{eqnarray*}
			& &\left(\bigg( t \otimes \displaystyle \bigotimes_{C \in \pi} \gamma_{C} \bigg) \square_{s} \bigg( u \otimes \displaystyle \bigotimes_{B \in \rho} \beta_{B} \bigg)\right) \square_{b} \Bigg( v \otimes \bigotimes_{A\in \tau} \alpha_A \Bigg)\\
           &=& \left( (t\circ_{C_{s}}u) \otimes \Big(\bigotimes_{C \in  \pi\setminus \{C_{s}\}}\gamma_{C}\Big) \otimes (\gamma_{C_{s}} \circ_s \beta_{B_{\ssmop{root}(u)}}) \otimes \Big(\bigotimes_{C \in  \rho \setminus \{B_{\ssmop{root}(u)}\}}\beta_{B}\Big) \right) \square_{b} \Bigg( v \otimes \bigotimes_{A\in \tau} \alpha_A \Bigg) \\
			&=&\Big((t\circ_{C_{s}}u)\circ_{B_{b}}v \Big) \otimes \Big(\bigotimes_{C \in  \pi\setminus \{C_{s}\}}\gamma_{C} \Big)\otimes \Big((\gamma_{C_{s}} \circ_s \beta_{B_{\ssmop{root}(u)}})\circ_{b} \alpha_{A_{\ssmop{root}(v)}}\Big) \otimes \Big(\bigotimes_{C \in  \rho \setminus \{B_{\ssmop{root}(u)}\}}\beta_{B}\Big) \\
   &&\hskip 100mm \otimes\Big(\bigotimes_{A\in \tau \setminus \{A_{\ssmop{root}(v)}\}} \alpha_A\Big)\\
			&=&\Big(t\circ_{C_{s}}(v\circ_{B_{\ssmop{root}(u)}u)} \Big) \otimes \Big(\bigotimes_{C \in  \pi\setminus \{C_{s}\}}\gamma_{C}\Big) \otimes \Big(\gamma_{C_{s}} \circ_{s} (\alpha_{A_{\ssmop{root}(v)}} \circ_b \beta_{B_{\ssmop{root}(u)}})\Big) \otimes \Big(\bigotimes_{C \in  \rho \setminus \{B_{\ssmop{root}(u)}\}}\beta_{B}\Big)\\
   &&\hskip 100mm\otimes \Big(\bigotimes_{A\in \tau \setminus \{A_{\ssmop{root}(v)}\}} \alpha_A\Big) \\
			&&\text{via nested associativity of $\mop{NAP}$ and $\tq$}\\
			&=&\Bigg( t \otimes \displaystyle \bigotimes_{C \in \pi} \gamma_{C} \Bigg) \square_{s} \left(\bigg( v \otimes \bigotimes_{A\in \tau} \alpha_A \bigg) \square_{b} \bigg( u \otimes \displaystyle \bigotimes_{B \in \rho} \beta_{B} \bigg)\right).
		\end{eqnarray*}
	- If $b\in B_b \ne B_r$, we have
	\begin{eqnarray*}
		& &\left(\bigg( t \otimes \displaystyle \bigotimes_{C \in \pi} \gamma_{C} \bigg) \square_{s} \bigg( u \otimes \displaystyle \bigotimes_{B \in \rho} \beta_{B} \bigg)\right) \square_{b} \Bigg( v \otimes \bigotimes_{A\in \tau} \alpha_A \Bigg)\\
        &=& \left( (t\circ_{C_{s}}u) \otimes \Big(\bigotimes_{C \in  \pi\setminus \{C_{s}\}}\gamma_{C}\Big) \otimes (\gamma_{C_{s}} \circ_s \beta_{B_{\ssmop{root}(u)}}) \otimes \Big(\bigotimes_{C \in  \rho \setminus \{B_{\ssmop{root}(u)}, B_b\}}\beta_{B}\Big) \right)  \square_{b} \Bigg( v \otimes \bigotimes_{A\in \tau} \alpha_A \Bigg) \\
		&=&\Big((t\circ_{C_{s}}u)\circ_{B_{b}}v \Big) \otimes \Big(\bigotimes_{C \in  \pi\setminus\{ C_{s}\}}\gamma_{C}\Big) \otimes (\gamma_{C_{s}} \circ_s \beta_{B_{\ssmop{root}(u)}})\otimes (\beta_{B_{b}}\circ_{b} \alpha_{A_{\ssmop{root}(v)}}\Big)\otimes \Big(\bigotimes_{C \in  \rho \setminus\{B_{\ssmop{root}(u)},B_b\}}\beta_{B}\Big)\\
        & &\otimes \Big(\bigotimes_{A\in \tau \setminus \{A_{\ssmop{root}(v)}\}} \alpha_A\Big)
        \end{eqnarray*}
        \noindent then
        \begin{eqnarray*}
        & &\left(\bigg( t \otimes \displaystyle \bigotimes_{C \in \pi} \gamma_{C} \bigg) \square_{s} \bigg( u \otimes \displaystyle \bigotimes_{B \in \rho} \beta_{B} \bigg)\right) \square_{b} \Bigg( v \otimes \bigotimes_{A\in \tau} \alpha_A \Bigg)\\
		&=&\Big(t\circ_{C_{s}}(v\circ_{B_{b}}u) \Big) \otimes \Big(\bigotimes_{C \in  \pi\setminus \{C_{s}\}}\gamma_{C}\Big) \otimes (\beta_{B_{b}} \circ_{b} \alpha_{A_{\ssmop{root}(v)}} )\otimes (\gamma_{C_{s}}\circ_s \beta_{B_{\ssmop{root}(u)}}) \otimes \Big(\bigotimes_{B \in  \rho \setminus \{B_{\ssmop{root}(u)},B_b\}}\beta_{B}\Big)\\
        & &\otimes \Big(\bigotimes_{A\in \tau \setminus \{A_{\ssmop{root}(v)}\}} \alpha_A\Big) \\
		&&\text{(via nested associativity of $\mop{NAP}$)}\\
		&=&\Bigg( t \otimes \displaystyle \bigotimes_{C \in \pi} \gamma_{C} \Bigg) \square_{s} \left(\bigg( v \otimes \bigotimes_{A\in \tau} \alpha_A \bigg) \square_{b} \bigg( u \otimes \displaystyle \bigotimes_{B \in \rho} \beta_{B} \bigg)\right).
	\end{eqnarray*}
 \end{itemize}
        \item Naturality:
        Let $I=S\sqcup T\sqcup R$,  $\pi \vdash S$ and $ \rho \vdash T$. Fix an element $s\in S$ and let $C_{s}$ the block of $\pi$ containing $s$. Consider $\sigma_1:S \to S'$ and $\sigma_2: T \to T'$ two bijections and let $\sigma:=\sigma_1\restr{S\setminus\{s\}} \sqcup \sigma_2$. \\
        If $(t \otimes \displaystyle \bigotimes_{C \in \pi} \gamma_C) \in \mop{NAP}\circ q [S] $ and $(u \otimes \displaystyle \bigotimes_{B \in \rho} \beta_B) \in \mop{NAP}\circ q [T] $, we have 
        \begin{eqnarray*}
            & &\big(\mop{NAP} \circ q\big)[\sigma]\left(\bigg(t \otimes \displaystyle \bigotimes_{C \in \pi} \gamma_C\bigg) \square_s \bigg(u \otimes \displaystyle \bigotimes_{B \in \rho} \beta_B\bigg) \right)\\&=& \big( \mop{NAP}\circ q\big)[\sigma]\left( (t\circ_{C_{s}}u) \otimes \Big(\bigotimes_{C \in  \pi\setminus \{C_{s}\}}\gamma_{C}\Big) \otimes (\gamma_{C_{s}} \circ_s \beta_{B_{\ssmop{root}(u)}}) \otimes \Big(\bigotimes_{C \in  \rho \setminus \{B_{\ssmop{root}(u)}, B_b\}}\beta_{B}\Big) \right)\\
            &=& \mop{NAP} [\sigma] (t\circ_{C_{s}}u) \otimes q[\sigma]\left( \Big(\bigotimes_{C \in  \pi\setminus \{C_{s}\}}\gamma_{C}\Big) \otimes (\gamma_{C_{s}} \circ_s \beta_{B_{\ssmop{root}(u)}}) \otimes \Big(\bigotimes_{C \in  \rho \setminus\{B_{\ssmop{root}(u)}, B_b\}}\beta_{B}\Big) \right)\\
            &=&\Big(NAP [\sigma_1]( t)\circ_{C_{s}} \mop{NAP} [\sigma_2]( u) \Big) \otimes\\
            & & \left( q[\sigma_1]\Big(\bigotimes_{C \in  \pi\setminus \{C_{s}\}}\gamma_{C}\Big) \otimes (q[\sigma_1](\gamma_{C_{s}}) \circ_s q[\sigma_2](\beta_{B_{\ssmop{root}(u)}}))  \otimes q[\sigma_2]\Big(\bigotimes_{C \in  \rho \setminus \{B_{\ssmop{root}(u)}, B_b\}}\beta_{B}\Big) \right)\\
            & &\text{(via ($N_1$) of $\mop{NAP}$ and q)}\\
            &=&\big(\mop{NAP} \circ q\big)[\sigma_1] \Bigg(t \otimes \displaystyle \bigotimes_{C \in \pi} \gamma_C\Bigg) \square_s \big(\mop{NAP} \circ q\big)[\sigma_2] \Bigg(u \otimes \displaystyle \bigotimes_{B \in \rho} \beta_B\Bigg) 
        \end{eqnarray*}
        \vspace{.1in}
         This proves $(N_{1})$. For $(N_{2})$, consider $s_{1}, s_{2} \in S$ and $\sigma : \{s_{1}\} \to \{s_{2}\}$ the function given by $\sigma(s_{1}) = s_{2}$. If $u_{s_{1}} \in (\mop{NAP} \circ \tq)[{s_{1}}] = \tq[\{s_{1}\}]$ and $u_{s_{2}}\in
   (\mop{NAP} \circ \tq)[\{s_{2}\}] = \tq[\{s_{2}\}]$, we have
   $$(\mop{NAP} \circ \tq)[\sigma](u_{s_{1}}) = \tq[\sigma](u_{s_{1}}) = u_{s_{2}},$$
   where the second equality is given by the naturality property $(N_{2})$ of $\circ$.
   \vspace{.1in}
   \item Unitality: let $s, t \in I.$ Since $(\mop{NAP} \circ \tq)[\{s\}] = \tq[\{s\}]$ and $(\mop{NAP} \circ \tq)[\{t\}] = \tq[\{t\}]$, the unitality properties $(U_{1})$ and $(U_{2})$ for $\square$ come directly from the ones for $\circ$.
	\end{itemize}
		\end{proof}

  \begin{remark}
      The $\mop{\sc Pre-Lie}$ operad does not satisfy property \eqref{asso-suppl}. It therefore seems impossible to define an operad structure on $\mop{\sc Pre-Lie} \circ \tq$ along these lines.
  \end{remark}  
  \section{\bf{The magmatic operad {\sc Mag}}}
  \subsection{Planar trees}\label{sec 6.1}
A planar binary tree is a finite directed tree with an insertion in the plane, such that all vertices have exactly two incoming edges and one outgoing edge. An edge can be internal (connecting two vertices) or external (with a free end). External edges are leaves. The root is the only edge that does not end in a vertex.
\begin{deft}
     Let $t_1$, $t_2 \in \mathcal{T}_{\smop{pl}}$. We define the planar tree $t_3:=t_1 \vee t_2$ by considering the unique planar binary tree with two leaves (the Y-shaped tree), replacing the left branch by $t_1$ and the right branch with $t_2$.
 \end{deft}
 The pair $(\mathcal{T}_{\smop{pl}},\vee)$ is the free magma generated by a single element, namely the one-edge tree. Knuth’s rotation correspondence between planar trees and planar rooted trees is defined by the bijection $\Phi:\mathcal{T}_{Pl} \to \mathcal{T}$ recursively given by $\Phi(|):=\bullet$ and $\Phi(t_1 \vee t_2) $ is the grafting of $t_1$ on the root of $t_2$ at the left.\\
 
Now let us denoting by $\mop{\sc Mag}$ the species of labeled planar rooted trees. For example,
 \[\mop{\sc Mag}[\{\emptyset\}]=\{0\}, \hskip 4mm \mop{\sc Mag}[\{*\}]=\left<\{\!\entoure *\}\right>,\hskip 4mm \mop{\sc Mag} [\{1,2\}]=\left<\left\{\stick 12,\ \stick 21\right\}\right>,\]
 \vskip 4mm
 \[\mop{\sc Mag} [\{1,2,3\}]=
  \left<\scalebox{0.75}{\ladder 123, \ladder 213,  \ladder 132,  \ladder 312, \ladder 231,\ladder 321, \hskip 3mm \cherry 123 ,\hskip 4mm \cherry 231, \hskip 4mm \cherry 312\hskip 4mm, \cherry 132, \hskip 4mm \cherry 213, \hskip 4mm \cherry 321\hskip 3mm}\right>.\]
 We are now ready to recall the operad structure on $\mop{\sc Mag}$:
\subsection{The $\mop{\sc Mag}$ operad}\label{sec 6.2}
 Let $I$ be a finite set, $I=S \sqcup T$, and $s \in S$. We define the partial composition as follows: $$\circ_{s}:\mop{\sc Mag}[S] \otimes \mop{\sc Mag}[T] \longrightarrow \mop{\sc Mag}[S \sqcup_{s} T],$$ such that for $u \in \mop{Mag}[S]$ and $v \in \mop{\sc Mag}[T]$, the partial composition $u \circ_{s} v$ is the tree of $\mop{\sc Mag}[S \sqcup_{s} T]$ obtained by replacing the vertex $s$ of $u$ by the tree $v$ and connecting each edge of $\mop{In}(s,u)$ at the root of $v$ by respecting the order the elements of $\mop{In}(r,v)$ are connecting in the left on the root $r$ of $v$, then we connect the elements of $\mop{In}(s,u)$. For example,\\

  \vskip 9mm
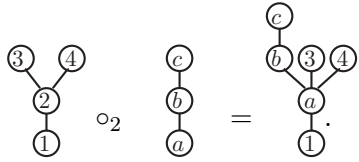
  \[\igrec 1234\hskip 4mm\circ_2\hskip 4mm\ladder abc\hskip 4mm =\hskip 4mm \treeones 1ab43c .\]
  \vskip 2mm

\
  
  \subsection{New operad structure on $\mop{\sc Mag}$}\label{sec 6.3}
In contrast with the $\mop{\sc NAP}$ operad, the $\mop{\sc Mag}$ operad does not verify \eqref{asso-suppl}.\\

Let $A$ and $B$ be two totally ordered sets. An $(A,B)$-shuffle is a permutation $\sigma$ of $A \sqcup B$ which preserves the total order of $A$ and $B$. Let $\mop{sh}(A,B)$ be the set of all $A,B$-shuffles. Its cardinal is \[|\mop{sh}(A,B)|={|A|+|B|\choose |A|}=\frac{(|A|+|B|)!}{|A|!|B|!}.\]
More generally, let $A_1,A_2,...,A_k$ be $k$ totally ordered sets. An $(A_1,A_2,...,A_k)$-shuffle is a permutation $\sigma$ of $A_1\sqcup A_2\sqcup ...\sqcup A_k$ which preserves the total order of each $A_i$, $i \in {1,2,...,k}$. The cardinal of the set $\mop{Sh}(A_1,A_2,...,A_k)$ of $(A_1,\ldots,A_k)$-shuffles is the multinomial coefficient 
\[\frac{(|A_1|+\cdots +|A_k|)!}{|A_1|!\cdots |A_k|!}.\]

Let $I=S \sqcup T$, let $t \in \mop{\sc Mag} [S]$ and $u \in \mop{\sc Mag} [T]$. Let $s$ be a fixed vertex of $t$. The partial composition
\[ \triangle_{s}:\mop{\sc Mag} [S] \otimes \mop{\sc Mag}[T] \longrightarrow \mop{\sc Mag} [S \sqcup_s T] \]
is defined as follows: for $t \in \mop{\sc Mag} [S]$ and $ u \in \mop{\sc Mag} [T]$ we have
$$t \triangle_s u := \displaystyle \sum_H t \circ_s^{\sc H} u$$
where $H$ runs over the set of $\big(\mop{In}(s,t), \mop{In}(r,u)\big)$-shuffles. Here $\mop{In(s,t)}$ is the set of branches arriving at the chosen vertex $s$ of $t$ ordered from left to right, and $\mop{In}(r,u)$ is the set of branches arriving at the root of $u$, also ordered from left to right. The tree $t \circ_s^{\small{H}} u\in \mop{\sc Mag} [S \sqcup_s T]$ is obtained by replacing the vertex $s$ of $t$ by the tree $u$ and connecting each edge in $H$ at the root of $u$ from left to right.

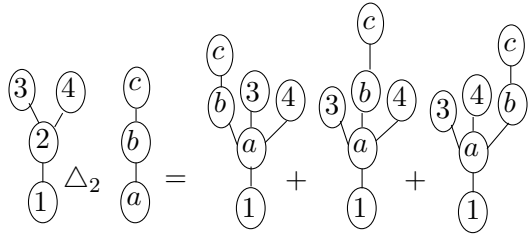
\begin{example}
\[
\begin{tikzpicture}[x=0.75pt,y=0.75pt,yscale=-1,xscale=1]

\draw   (36.75,139.01) .. controls (36.75,133.29) and (39.99,128.66) .. (43.99,128.66) .. controls (47.99,128.66) and (51.23,133.29) .. (51.23,139.01) .. controls (51.23,144.73) and (47.99,149.37) .. (43.99,149.37) .. controls (39.99,149.37) and (36.75,144.73) .. (36.75,139.01) -- cycle ;
\draw   (26.92,82.34) .. controls (26.92,76.62) and (29.85,71.98) .. (33.47,71.98) .. controls (37.09,71.98) and (40.02,76.62) .. (40.02,82.34) .. controls (40.02,88.06) and (37.09,92.7) .. (33.47,92.7) .. controls (29.85,92.7) and (26.92,88.06) .. (26.92,82.34) -- cycle ;
\draw   (50.83,83.49) .. controls (50.83,77.77) and (54.07,73.13) .. (58.07,73.13) .. controls (62.07,73.13) and (65.32,77.77) .. (65.32,83.49) .. controls (65.32,89.21) and (62.07,93.85) .. (58.07,93.85) .. controls (54.07,93.85) and (50.83,89.21) .. (50.83,83.49) -- cycle ;
\draw   (37.15,108.81) .. controls (37.15,103.09) and (40.39,98.45) .. (44.39,98.45) .. controls (48.39,98.45) and (51.63,103.09) .. (51.63,108.81) .. controls (51.63,114.53) and (48.39,119.16) .. (44.39,119.16) .. controls (40.39,119.16) and (37.15,114.53) .. (37.15,108.81) -- cycle ;
\draw    (44,129) -- (44,119) ;
\draw    (50,101) -- (54,94) ;
\draw    (37,89) -- (42,99) ;
\draw   (83.32,139.01) .. controls (83.32,133.29) and (86.56,128.66) .. (90.56,128.66) .. controls (94.56,128.66) and (97.81,133.29) .. (97.81,139.01) .. controls (97.81,144.73) and (94.56,149.37) .. (90.56,149.37) .. controls (86.56,149.37) and (83.32,144.73) .. (83.32,139.01) -- cycle ;
\draw   (84.78,81.19) .. controls (84.78,75.47) and (87.71,70.83) .. (91.33,70.83) .. controls (94.94,70.83) and (97.88,75.47) .. (97.88,81.19) .. controls (97.88,86.91) and (94.94,91.54) .. (91.33,91.54) .. controls (87.71,91.54) and (84.78,86.91) .. (84.78,81.19) -- cycle ;
\draw   (83.72,108.81) .. controls (83.72,103.09) and (86.97,98.45) .. (90.97,98.45) .. controls (94.97,98.45) and (98.21,103.09) .. (98.21,108.81) .. controls (98.21,114.53) and (94.97,119.16) .. (90.97,119.16) .. controls (86.97,119.16) and (83.72,114.53) .. (83.72,108.81) -- cycle ;
\draw    (91,129) -- (91,119) ;
\draw    (91,98) -- (91,92) ;
\draw   (141.61,142.01) .. controls (141.61,136.29) and (144.85,131.66) .. (148.86,131.66) .. controls (152.86,131.66) and (156.1,136.29) .. (156.1,142.01) .. controls (156.1,147.73) and (152.86,152.37) .. (148.86,152.37) .. controls (144.85,152.37) and (141.61,147.73) .. (141.61,142.01) -- cycle ;
\draw   (127.42,91.09) .. controls (127.42,85.37) and (130.35,80.74) .. (133.97,80.74) .. controls (137.59,80.74) and (140.52,85.37) .. (140.52,91.09) .. controls (140.52,96.81) and (137.59,101.45) .. (133.97,101.45) .. controls (130.35,101.45) and (127.42,96.81) .. (127.42,91.09) -- cycle ;
\draw   (143.69,83.04) .. controls (143.69,77.32) and (146.93,72.68) .. (150.93,72.68) .. controls (154.93,72.68) and (158.18,77.32) .. (158.18,83.04) .. controls (158.18,88.76) and (154.93,93.39) .. (150.93,93.39) .. controls (146.93,93.39) and (143.69,88.76) .. (143.69,83.04) -- cycle ;
\draw   (142.01,111.81) .. controls (142.01,106.09) and (145.26,101.45) .. (149.26,101.45) .. controls (153.26,101.45) and (156.5,106.09) .. (156.5,111.81) .. controls (156.5,117.53) and (153.26,122.16) .. (149.26,122.16) .. controls (145.26,122.16) and (142.01,117.53) .. (142.01,111.81) -- cycle ;
\draw    (149,131) -- (149,121) ;
\draw    (149,101) -- (149,93) ;
\draw    (138,99) -- (142,110) ;
\draw   (126.72,64.87) .. controls (126.72,59.15) and (129.65,54.51) .. (133.27,54.51) .. controls (136.88,54.51) and (139.82,59.15) .. (139.82,64.87) .. controls (139.82,70.59) and (136.88,75.23) .. (133.27,75.23) .. controls (129.65,75.23) and (126.72,70.59) .. (126.72,64.87) -- cycle ;
\draw   (161.7,87.64) .. controls (161.7,81.92) and (164.94,77.28) .. (168.94,77.28) .. controls (172.94,77.28) and (176.19,81.92) .. (176.19,87.64) .. controls (176.19,93.36) and (172.94,98) .. (168.94,98) .. controls (164.94,98) and (161.7,93.36) .. (161.7,87.64) -- cycle ;
\draw    (167,98) -- (156,109) ;
\draw    (134,81) -- (134,75) ;
\draw   (197.61,142.01) .. controls (197.61,136.29) and (200.85,131.66) .. (204.86,131.66) .. controls (208.86,131.66) and (212.1,136.29) .. (212.1,142.01) .. controls (212.1,147.73) and (208.86,152.37) .. (204.86,152.37) .. controls (200.85,152.37) and (197.61,147.73) .. (197.61,142.01) -- cycle ;
\draw   (183.42,91.09) .. controls (183.42,85.37) and (186.35,80.74) .. (189.97,80.74) .. controls (193.59,80.74) and (196.52,85.37) .. (196.52,91.09) .. controls (196.52,96.81) and (193.59,101.45) .. (189.97,101.45) .. controls (186.35,101.45) and (183.42,96.81) .. (183.42,91.09) -- cycle ;
\draw   (199.69,83.04) .. controls (199.69,77.32) and (202.93,72.68) .. (206.93,72.68) .. controls (210.93,72.68) and (214.18,77.32) .. (214.18,83.04) .. controls (214.18,88.76) and (210.93,93.39) .. (206.93,93.39) .. controls (202.93,93.39) and (199.69,88.76) .. (199.69,83.04) -- cycle ;
\draw   (198.01,111.81) .. controls (198.01,106.09) and (201.26,101.45) .. (205.26,101.45) .. controls (209.26,101.45) and (212.5,106.09) .. (212.5,111.81) .. controls (212.5,117.53) and (209.26,122.16) .. (205.26,122.16) .. controls (201.26,122.16) and (198.01,117.53) .. (198.01,111.81) -- cycle ;
\draw    (205,132) -- (205,122) ;
\draw    (205,101) -- (205,94) ;
\draw    (194,98) -- (198,110) ;
\draw   (217.7,87.64) .. controls (217.7,81.92) and (220.94,77.28) .. (224.94,77.28) .. controls (228.94,77.28) and (232.19,81.92) .. (232.19,87.64) .. controls (232.19,93.36) and (228.94,98) .. (224.94,98) .. controls (220.94,98) and (217.7,93.36) .. (217.7,87.64) -- cycle ;
\draw    (223,98) -- (212,108) ;
\draw   (253.61,144.01) .. controls (253.61,138.29) and (256.85,133.66) .. (260.86,133.66) .. controls (264.86,133.66) and (268.1,138.29) .. (268.1,144.01) .. controls (268.1,149.73) and (264.86,154.37) .. (260.86,154.37) .. controls (256.85,154.37) and (253.61,149.73) .. (253.61,144.01) -- cycle ;
\draw   (239.42,93.09) .. controls (239.42,87.37) and (242.35,82.74) .. (245.97,82.74) .. controls (249.59,82.74) and (252.52,87.37) .. (252.52,93.09) .. controls (252.52,98.81) and (249.59,103.45) .. (245.97,103.45) .. controls (242.35,103.45) and (239.42,98.81) .. (239.42,93.09) -- cycle ;
\draw   (255.69,85.04) .. controls (255.69,79.32) and (258.93,74.68) .. (262.93,74.68) .. controls (266.93,74.68) and (270.18,79.32) .. (270.18,85.04) .. controls (270.18,90.76) and (266.93,95.39) .. (262.93,95.39) .. controls (258.93,95.39) and (255.69,90.76) .. (255.69,85.04) -- cycle ;
\draw   (254.01,113.81) .. controls (254.01,108.09) and (257.26,103.45) .. (261.26,103.45) .. controls (265.26,103.45) and (268.5,108.09) .. (268.5,113.81) .. controls (268.5,119.53) and (265.26,124.16) .. (261.26,124.16) .. controls (257.26,124.16) and (254.01,119.53) .. (254.01,113.81) -- cycle ;
\draw    (261,134) -- (261,124) ;
\draw    (261,103) -- (261,96) ;
\draw    (250,100) -- (254,112) ;
\draw   (273.7,89.64) .. controls (273.7,83.92) and (276.94,79.28) .. (280.94,79.28) .. controls (284.94,79.28) and (288.19,83.92) .. (288.19,89.64) .. controls (288.19,95.36) and (284.94,100) .. (280.94,100) .. controls (276.94,100) and (273.7,95.36) .. (273.7,89.64) -- cycle ;
\draw    (279,99) -- (268,112) ;
\draw   (202.72,48.87) .. controls (202.72,43.15) and (205.65,38.51) .. (209.27,38.51) .. controls (212.88,38.51) and (215.82,43.15) .. (215.82,48.87) .. controls (215.82,54.59) and (212.88,59.23) .. (209.27,59.23) .. controls (205.65,59.23) and (202.72,54.59) .. (202.72,48.87) -- cycle ;
\draw    (209,72) -- (209,59) ;
\draw   (273.72,59.87) .. controls (273.72,54.15) and (276.65,49.51) .. (280.27,49.51) .. controls (283.88,49.51) and (286.82,54.15) .. (286.82,59.87) .. controls (286.82,65.59) and (283.88,70.23) .. (280.27,70.23) .. controls (276.65,70.23) and (273.72,65.59) .. (273.72,59.87) -- cycle ;
\draw    (280,79) -- (280,70) ;

\draw (38,133) node [anchor=north west][inner sep=0.75pt]    {$1$};
\draw (39,101) node [anchor=north west][inner sep=0.75pt]    {$2$};
\draw (28,75) node [anchor=north west][inner sep=0.75pt]    {$3$};
\draw (52,75) node [anchor=north west][inner sep=0.75pt]    {$4$};
\draw (53,120) node [anchor=north west][inner sep=0.75pt]    {$\triangle _{2}$};
\draw (85,133) node [anchor=north west][inner sep=0.75pt]    {$a$};
\draw (85,101) node [anchor=north west][inner sep=0.75pt]    {$b$};
\draw (86,75) node [anchor=north west][inner sep=0.75pt]    {$c$};
\draw (104,124) node [anchor=north west][inner sep=0.75pt]    {$=$};
\draw (143,136) node [anchor=north west][inner sep=0.75pt]    {$1$};
\draw (129,82) node [anchor=north west][inner sep=0.75pt]    {$b$};
\draw (145,77) node [anchor=north west][inner sep=0.75pt]    {$3$};
\draw (163,80) node [anchor=north west][inner sep=0.75pt]    {$4$};
\draw (128,59) node [anchor=north west][inner sep=0.75pt]    {$c$};
\draw (143,106) node [anchor=north west][inner sep=0.75pt]    {$a$};
\draw (199,136) node [anchor=north west][inner sep=0.75pt]    {$1$};
\draw (185,84) node [anchor=north west][inner sep=0.75pt]    {$3$};
\draw (202,77) node [anchor=north west][inner sep=0.75pt]    {$b$};
\draw (219,80) node [anchor=north west][inner sep=0.75pt]    {$4$};
\draw (199,106) node [anchor=north west][inner sep=0.75pt]    {$a$};
\draw (256,137) node [anchor=north west][inner sep=0.75pt]    {$1$};
\draw (241,86) node [anchor=north west][inner sep=0.75pt]    {$3$};
\draw (258,79) node [anchor=north west][inner sep=0.75pt]    {$4$};
\draw (275,82) node [anchor=north west][inner sep=0.75pt]    {$b$};
\draw (255,108) node [anchor=north west][inner sep=0.75pt]    {$a$};
\draw (204,44) node [anchor=north west][inner sep=0.75pt]    {$c$};
\draw (275,55) node [anchor=north west][inner sep=0.75pt]    {$c$};
\draw (165,122) node [anchor=north west][inner sep=0.75pt]    {$+$};
\draw (225,122) node [anchor=north west][inner sep=0.75pt]    {$+$};
\end{tikzpicture}
\]
\end{example}

\medbreak

\begin{theorem}
    $(\mop{\sc Mag},\triangle)$ is an operad, which moreover verifies the additional condition \eqref{asso-suppl}.
\end{theorem}
\begin{proof}
    Let $I=S \sqcup T \sqcup R$, let $t \in \mop{\sc Mag} [S]$, let $u \in \mop{\sc Mag} [T]$ and $v \in \mop{\sc Mag} [R]$. Let $s$, $s'$ be two distinct fixed vertices of $t$ and $m$ a fixed vertices of $u$.
    Let $H$ be a $\big(\mop{In}(s,t), \mop{In}(r,u)\big)$-shuffle, $H'$ be a $\big(\mop{In}(s',t), \mop{In}(r,v)\big)$-shuffle and $K$ be a $\big(\mop{In}(m,u), \mop{In}(r,v)\big)$-shuffle. 
    \begin{itemize}
        \item Parallel associativity:
        \begin{eqnarray*}
            & &(t \triangle_s u) \triangle_{s'} v \\&=& \Big( \displaystyle \sum_{H\in\smop{Sh}(\smop{In}(s,t), \,\smop{In}(r,u))} t \circ_s^H u \Big) \triangle_{s'} v\\
            &=&\displaystyle \sum_{H'\in\smop{Sh}(\smop{In}(s',t), \,\smop{In}(r,v))} \ \displaystyle \sum_{H\in\smop{Sh}(\smop{In}(s,t), \,\smop{In}(r,u))} \big((t \circ_s^H u)\circ_{s'}^{H'}v \big)\\
            &=& \displaystyle \sum_{H\in\smop{Sh}(\smop{In}(s,t), \,\smop{In}(r,u))} \ \displaystyle \sum_{H'\in\smop{Sh}(\smop{In}(s',t),\, \smop{In}(r,v))} \big((t \circ_{s'}^{H'}v)\circ_s^H u \big)\\
            &&\hbox{(by independence of shuffle operations at $s$ and $s'$)}\\
            &=&(t \triangle_{s'} v) \triangle_s u.
        \end{eqnarray*}
        \item  Nested associativity. Two subcases must be considered.
        \begin{itemize}
            \item If $m$ is the root of the tree $u$,
            \begin{eqnarray*}
                (t \triangle_s u) \triangle_{m} v &=& \Big( \displaystyle \sum_{H\in\smop{Sh}(\smop{In}(s,t), \,\smop{In}(m,u))} t \circ_s^H u \Big) \triangle_{m} v\\
                &=&\displaystyle \sum_{H\in\smop{Sh}(\smop{In}(s,t),\, \smop{In}(m,u))}\ \displaystyle \sum_{K\in\smop{Sh}(\smop{In}(m,t \circ_s^H u), \,\smop{In}(r,v))} (t \circ_s^H u)\circ_{m}^{K}v \\
                &=& \displaystyle \sum_{K'\in\smop{Sh}(\smop{In}(s,t),\,\smop{In}(m,u),\, \smop{In}(r,v))} (t\circ_s u\circ_m v)^{K'}
            \end{eqnarray*}
            where $(t\circ_s u\circ_m v)^{K'}$ is obtained from $t$, $u$ and $v$ by merging the three vertices $s,m,r$ into a single one, and by ordering the branches arriving at this vertex according to the shuffle $K'$. On the other hand, we have
            \begin{eqnarray*}
            t \triangle_s (u \triangle_{m} v) &=& t\triangle_s\Big(\sum_{L\in\smop{Sh}(\smop{In}(m,u),\,\smop{In}(r,v))}u\circ_m^L v\Big)\\
            &=& \sum_{L\in\smop{Sh}(\smop{In}(m,u),\,\smop{In}(r,v))}\ \sum_{M\in\smop{Sh}(\smop{In}(s,t),\,\smop{In}(r, u\circ_m^L v))}t\circ_s^M (u\circ_m^L v)\\
            &=& \displaystyle \sum_{K'\in\smop{Sh}\big(\smop{In}(s,t),\,\smop{In}(m,u), \,\smop{In}(r,v)\big)} (t\circ_s u\circ_m v)^{K'}\\
            &=& (t \triangle_s u) \triangle_{m} v.
            \end{eqnarray*}
            \item If $m$ is different from the root of $u$,
        \begin{eqnarray*}
           & & (t \triangle_s u) \triangle_{m} v\\ &=& \Big( \displaystyle \sum_{H\in\smop{Sh}(\smop{In}(s,t),\, \smop{In}(r,u))} t \circ_s^H u \Big) \triangle_{m} v\\
           &=&\displaystyle \sum_{K\in\smop{Sh}(\smop{In}(m,u), \,\smop{In}(r,v))} \ \displaystyle \sum_{H\in\smop{Sh}(\smop{In}(s,t), \,\smop{In}(r,u))} \big((t \circ_s^H u)\circ_{m}^{K}v \big)\\
           &=&\displaystyle \sum_{H\in\smop{Sh}(\smop{In}(s,t), \,\smop{In}(r,u))}\  \displaystyle \sum_{K\in\smop{Sh}(\smop{In}(m,u), \,\smop{In}(r,v))}  \big(t \circ_s^H (u\circ_{m}^{K} v) \big)\\
           &=& t \triangle_s (u \triangle_{m}v).
        \end{eqnarray*}
        \end{itemize}
        \item Property \eqref{asso-suppl} is easily checked, by shuffling the branches above $s$ coming from the three trees considered.
    \end{itemize} 
\end{proof}

\newpage
  \section{\bf{Operad structure on $\mop{\sc Mag} \circ \tq$}}
  \noindent From the substitution formula
	$$\mop{\sc Mag} \circ \tq [I]=\displaystyle \bigoplus_{\pi \vdash I} \mop{\sc Mag} [\pi] \otimes \tq(\pi),$$
	one can see that a typical element in $\mop{\sc Mag} \circ \tq [I]$ is of the form $t \otimes \displaystyle \bigotimes_{C \in \pi}\gamma_C$, where $\pi \vdash I$ and $t \in \mop{Mag}[\pi]$, and where $\gamma_C\in \tq[C]$ for any block $C$ of $\pi$.
 
    \subsection{Diamond operation}
Let $I=S \sqcup T$. The partial composition
\[ \diamondsuit_{s}:\mop{\sc Mag} \circ \tq[S] \otimes \mop{\sc Mag} \circ \tq[T] \longrightarrow \mop{\sc Mag} \circ \tq[S \sqcup_s T] \]
is defined as follows, let $\pi \vdash S$, let $ \rho \vdash T$, let $s \in S$ and let $C_s$ be the block of $\pi$ which contains $s$, for $\bigg(t\otimes \displaystyle \bigotimes_{c\in \pi} \alpha_{c}\bigg) \in \mop{\sc Mag} \circ \tq[S]$ and $\bigg(u\otimes \displaystyle \bigotimes_{B\in \rho} \beta_{B}\bigg) \in \mop{\sc Mag} \circ \tq[T]$, we have
\begin{eqnarray*}
    &&\bigg(t\otimes \displaystyle \bigotimes_{c\in \pi} \alpha_{c}\bigg) \diamondsuit_{s} \bigg(u\otimes \displaystyle \bigotimes_{B\in \rho} \beta_{B}\bigg)\\
    &:=& \sum_{H \in \smop{sh}\big(\smop{In}(C_s,t),\,\smop{In}(B_r,u)\big)}(t \circ_{\smop{C}_s}^{\smop{H}} u) \otimes (\alpha_{c_s} \circ_s \beta_{\smop{B}_r}) \otimes 
   \bigg(\bigotimes_{C'\in \pi \setminus\{C_s\}} \alpha_{C'}\bigg)   \otimes \bigg( \bigotimes_{B'\in \rho \setminus_{\{B_r\}}} \beta_{B'}\bigg).
\end{eqnarray*}
    
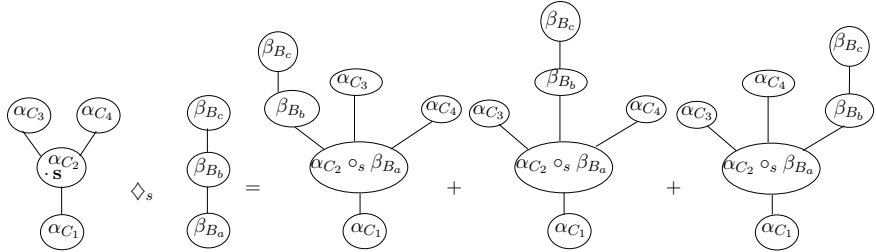
\begin{example}

\

\scalebox{0.7}{
   \begin{tikzpicture}[x=0.75pt,y=0.75pt,yscale=-1,xscale=1]

\draw   (57.75,212.83) .. controls (57.75,205.85) and (64.72,200.19) .. (73.32,200.19) .. controls (81.91,200.19) and (88.88,205.85) .. (88.88,212.83) .. controls (88.88,219.82) and (81.91,225.48) .. (73.32,225.48) .. controls (64.72,225.48) and (57.75,219.82) .. (57.75,212.83) -- cycle ;
\draw    (73,180) -- (73,200) ;
\draw   (55.3,166.24) .. controls (55.3,158.34) and (63.18,151.93) .. (72.91,151.93) .. controls (82.63,151.93) and (90.52,158.34) .. (90.52,166.24) .. controls (90.52,174.15) and (82.63,180.55) .. (72.91,180.55) .. controls (63.18,180.55) and (55.3,174.15) .. (55.3,166.24) -- cycle ;
\draw    (86,158) -- (98,140) ;
\draw   (34,128.64) .. controls (34,121.65) and (40.97,115.99) .. (49.56,115.99) .. controls (58.16,115.99) and (65.13,121.65) .. (65.13,128.64) .. controls (65.13,135.62) and (58.16,141.28) .. (49.56,141.28) .. controls (40.97,141.28) and (34,135.62) .. (34,128.64) -- cycle ;
\draw   (83.97,128.64) .. controls (83.97,121.65) and (90.93,115.99) .. (99.53,115.99) .. controls (108.13,115.99) and (115.09,121.65) .. (115.09,128.64) .. controls (115.09,135.62) and (108.13,141.28) .. (99.53,141.28) .. controls (90.93,141.28) and (83.97,135.62) .. (83.97,128.64) -- cycle ;
\draw   (163.42,211.67) .. controls (163.42,204.78) and (170.3,199.19) .. (178.78,199.19) .. controls (187.26,199.19) and (194.14,204.78) .. (194.14,211.67) .. controls (194.14,218.56) and (187.26,224.15) .. (178.78,224.15) .. controls (170.3,224.15) and (163.42,218.56) .. (163.42,211.67) -- cycle ;
\draw   (163.63,126.81) .. controls (163.63,119.92) and (170.5,114.33) .. (178.98,114.33) .. controls (187.47,114.33) and (194.34,119.92) .. (194.34,126.81) .. controls (194.34,133.7) and (187.47,139.29) .. (178.98,139.29) .. controls (170.5,139.29) and (163.63,133.7) .. (163.63,126.81) -- cycle ;
\draw   (163.42,167.74) .. controls (163.42,160.85) and (170.3,155.26) .. (178.78,155.26) .. controls (187.26,155.26) and (194.14,160.85) .. (194.14,167.74) .. controls (194.14,174.63) and (187.26,180.22) .. (178.78,180.22) .. controls (170.3,180.22) and (163.42,174.63) .. (163.42,167.74) -- cycle ;
\draw    (178,138) -- (178,155) ;
\draw    (178,180) -- (178,199) ;
\draw   (275.76,212.07) .. controls (275.76,205.08) and (282.73,199.42) .. (291.32,199.42) .. controls (299.92,199.42) and (306.89,205.08) .. (306.89,212.07) .. controls (306.89,219.05) and (299.92,224.71) .. (291.32,224.71) .. controls (282.73,224.71) and (275.76,219.05) .. (275.76,212.07) -- cycle ;
\draw    (288,185) -- (288,200) ;
\draw   (251.97,166.06) .. controls (251.97,155.91) and (267.73,147.67) .. (287.17,147.67) .. controls (306.61,147.67) and (322.37,155.91) .. (322.37,166.06) .. controls (322.37,176.22) and (306.61,184.45) .. (287.17,184.45) .. controls (267.73,184.45) and (251.97,176.22) .. (251.97,166.06) -- cycle ;
\draw    (310,153) -- (340,133) ;
\draw    (241,137) -- (262,153) ;
\draw   (219.24,123.78) .. controls (219.24,116.51) and (228.08,110.62) .. (238.98,110.62) .. controls (249.88,110.62) and (258.72,116.51) .. (258.72,123.78) .. controls (258.72,131.05) and (249.88,136.95) .. (238.98,136.95) .. controls (228.08,136.95) and (219.24,131.05) .. (219.24,123.78) -- cycle ;
\draw   (332.31,124.18) .. controls (332.31,118.82) and (338.72,114.48) .. (346.63,114.48) .. controls (354.54,114.48) and (360.95,118.82) .. (360.95,124.18) .. controls (360.95,129.54) and (354.54,133.88) .. (346.63,133.88) .. controls (338.72,133.88) and (332.31,129.54) .. (332.31,124.18) -- cycle ;
\draw   (214.75,82.66) .. controls (214.75,74.75) and (221.07,68.34) .. (228.88,68.34) .. controls (236.69,68.34) and (243.01,74.75) .. (243.01,82.66) .. controls (243.01,90.57) and (236.69,96.99) .. (228.88,96.99) .. controls (221.07,96.99) and (214.75,90.57) .. (214.75,82.66) -- cycle ;
\draw   (266.43,104.64) .. controls (266.43,99.35) and (274.85,95.07) .. (285.24,95.07) .. controls (295.63,95.07) and (304.05,99.35) .. (304.05,104.64) .. controls (304.05,109.93) and (295.63,114.22) .. (285.24,114.22) .. controls (274.85,114.22) and (266.43,109.93) .. (266.43,104.64) -- cycle ;
\draw    (229,97) -- (229,112) ;
\draw    (284,114) -- (284,147) ;
\draw   (423.32,212.07) .. controls (423.32,205.08) and (430.29,199.42) .. (438.88,199.42) .. controls (447.48,199.42) and (454.45,205.08) .. (454.45,212.07) .. controls (454.45,219.05) and (447.48,224.71) .. (438.88,224.71) .. controls (430.29,224.71) and (423.32,219.05) .. (423.32,212.07) -- cycle ;
\draw    (435,184) -- (435,200) ;
\draw   (399.53,166.06) .. controls (399.53,155.91) and (415.29,147.67) .. (434.73,147.67) .. controls (454.17,147.67) and (469.93,155.91) .. (469.93,166.06) .. controls (469.93,176.22) and (454.17,184.45) .. (434.73,184.45) .. controls (415.29,184.45) and (399.53,176.22) .. (399.53,166.06) -- cycle ;
\draw    (458,151) -- (488,133) ;
\draw    (388,135) -- (409,153) ;
\draw   (366.8,127.24) .. controls (366.8,121.88) and (373.21,117.54) .. (381.12,117.54) .. controls (389.03,117.54) and (395.44,121.88) .. (395.44,127.24) .. controls (395.44,132.6) and (389.03,136.95) .. (381.12,136.95) .. controls (373.21,136.95) and (366.8,132.6) .. (366.8,127.24) -- cycle ;
\draw   (479.87,124.18) .. controls (479.87,118.82) and (486.28,114.48) .. (494.19,114.48) .. controls (502.1,114.48) and (508.51,118.82) .. (508.51,124.18) .. controls (508.51,129.54) and (502.1,133.88) .. (494.19,133.88) .. controls (486.28,133.88) and (479.87,129.54) .. (479.87,124.18) -- cycle ;
\draw   (418.25,60.82) .. controls (418.25,52.91) and (424.57,46.5) .. (432.38,46.5) .. controls (440.19,46.5) and (446.51,52.91) .. (446.51,60.82) .. controls (446.51,68.73) and (440.19,75.15) .. (432.38,75.15) .. controls (424.57,75.15) and (418.25,68.73) .. (418.25,60.82) -- cycle ;
\draw   (413.99,104.64) .. controls (413.99,99.35) and (422.41,95.07) .. (432.8,95.07) .. controls (443.19,95.07) and (451.61,99.35) .. (451.61,104.64) .. controls (451.61,109.93) and (443.19,114.22) .. (432.8,114.22) .. controls (422.41,114.22) and (413.99,109.93) .. (413.99,104.64) -- cycle ;
\draw    (432,75) -- (432,95) ;
\draw    (432,114) -- (432,147) ;
\draw   (572.81,213.02) .. controls (572.81,206.03) and (579.78,200.37) .. (588.37,200.37) .. controls (596.97,200.37) and (603.94,206.03) .. (603.94,213.02) .. controls (603.94,220) and (596.97,225.66) .. (588.37,225.66) .. controls (579.78,225.66) and (572.81,220) .. (572.81,213.02) -- cycle ;
\draw    (585,186) -- (585,200) ;
\draw   (549.02,167.01) .. controls (549.02,156.86) and (564.78,148.62) .. (584.22,148.62) .. controls (603.66,148.62) and (619.42,156.86) .. (619.42,167.01) .. controls (619.42,177.17) and (603.66,185.4) .. (584.22,185.4) .. controls (564.78,185.4) and (549.02,177.17) .. (549.02,167.01) -- cycle ;
\draw    (607,154) -- (637,136) ;
\draw    (538,136) -- (559,154) ;
\draw   (516.29,128.19) .. controls (516.29,122.83) and (522.7,118.49) .. (530.61,118.49) .. controls (538.52,118.49) and (544.93,122.83) .. (544.93,128.19) .. controls (544.93,133.55) and (538.52,137.9) .. (530.61,137.9) .. controls (522.7,137.9) and (516.29,133.55) .. (516.29,128.19) -- cycle ;
\draw   (623.28,125.16) .. controls (623.28,118.5) and (631.05,113.11) .. (640.64,113.11) .. controls (650.23,113.11) and (658,118.5) .. (658,125.16) .. controls (658,131.81) and (650.23,137.21) .. (640.64,137.21) .. controls (631.05,137.21) and (623.28,131.81) .. (623.28,125.16) -- cycle ;
\draw   (626.57,78.86) .. controls (626.57,70.95) and (632.9,64.54) .. (640.7,64.54) .. controls (648.51,64.54) and (654.83,70.95) .. (654.83,78.86) .. controls (654.83,86.78) and (648.51,93.19) .. (640.7,93.19) .. controls (632.9,93.19) and (626.57,86.78) .. (626.57,78.86) -- cycle ;
\draw   (563.48,105.59) .. controls (563.48,100.3) and (571.9,96.02) .. (582.29,96.02) .. controls (592.68,96.02) and (601.1,100.3) .. (601.1,105.59) .. controls (601.1,110.88) and (592.68,115.17) .. (582.29,115.17) .. controls (571.9,115.17) and (563.48,110.88) .. (563.48,105.59) -- cycle ;
\draw    (640.7,93.19) -- (640.51,113.11) ;
\draw    (582,115) -- (582,148) ;
\draw    (45,140) -- (58,158) ;

\draw (62,206) node [anchor=north west][inner sep=0.75pt]    {$\alpha _{C_1}$};
\draw (62,156) node [anchor=north west][inner sep=0.75pt]    {$\alpha _{C_2}$};
\draw (59,167) node [anchor=north west][inner sep=0.75pt]    {$\cdot \,\mathbf s$};
\draw (86,121) node [anchor=north west][inner sep=0.75pt]    {$\alpha _{C_4}$};
\draw (37,121) node [anchor=north west][inner sep=0.75pt]    {$\alpha _{C_3}$};
\draw (121,176) node [anchor=north west][inner sep=0.75pt]    {$\diamondsuit _{s}$};
\draw (167,201) node [anchor=north west][inner sep=0.75pt]    {$\beta _{B_a}$};
\draw (167,158) node [anchor=north west][inner sep=0.75pt]    {$\beta _{B_b}$};
\draw (167,116) node [anchor=north west][inner sep=0.75pt]    {$\beta _{B_c}$};
\draw (204,178) node [anchor=north west][inner sep=0.75pt]    {$=$};
\draw (280,205) node [anchor=north west][inner sep=0.75pt]    {$\alpha _{C_1}$};
\draw (251,155) node [anchor=north west][inner sep=0.75pt]   {$\alpha _{C_2} \circ _{s} \beta _{B_a}$};
\draw (270,98) node [anchor=north west][inner sep=0.75pt]    {$\alpha _{C_3}$};
\draw (334,117) node [anchor=north west][inner sep=0.75pt]    {$\alpha _{C_4}$};
\draw (226,115) node [anchor=north west][inner sep=0.75pt]    {$\beta _{B_b}$};
\draw (216,72) node [anchor=north west][inner sep=0.75pt]    {$\beta _{B_c}$};
\draw (427,205) node [anchor=north west][inner sep=0.75pt]    {$\alpha _{C_1}$};
\draw (399,155) node [anchor=north west][inner sep=0.75pt]   {$\alpha _{C_2} \circ _{s} \beta _{B_a}$};
\draw (368,120) node [anchor=north west][inner sep=0.75pt]    {$\alpha _{C_3}$};
\draw (481,117) node [anchor=north west][inner sep=0.75pt]    {$\alpha _{C_4}$};
\draw (421,94) node [anchor=north west][inner sep=0.75pt]    {$\beta _{B_b}$};
\draw (420,52) node [anchor=north west][inner sep=0.75pt]    {$\beta _{B_c}$};
\draw (576,205) node [anchor=north west][inner sep=0.75pt]    {$\alpha _{C_1}$};
\draw (548,156) node [anchor=north west][inner sep=0.75pt]   {$\alpha _{C_2} \circ _{s} \beta _{B_a}$};
\draw (518,121) node [anchor=north west][inner sep=0.75pt]    {$\alpha _{C_3}$};
\draw (571,98) node [anchor=north west][inner sep=0.75pt]    {$\alpha _{C_4}$};
\draw (628,116) node [anchor=north west][inner sep=0.75pt]    {$\beta _{B_b}$};
\draw (627.61,68.13) node [anchor=north west][inner sep=0.75pt]    {$\beta _{B_c}$};
\draw (349,175) node [anchor=north west][inner sep=0.75pt]    {$+$};
\draw (507,175) node [anchor=north west][inner sep=0.75pt]    {$+$};
\end{tikzpicture}
}
\end{example}

\begin{theorem}
    $(\mop{\sc Mag} \circ \tq, \diamondsuit)$ is an operad.
\end{theorem}
\begin{proof} 
    \item Associativity: let $I=S\sqcup T\sqcup R$,  let $\pi \vdash S$, $ \rho \vdash T$ and $\tau \vdash R$,  let $s,s' \in S$ and let $t \in T$. Let $C_{s}$, $C_{s'}$ be the blocks of $\pi$ which contain respectively $s$ and $s'$, and let $B_t$ be the block of $\rho$ which contains $t$.

            \begin{itemize}
		      \item Parallel associativity: two subcases must be considered. 
        
		      \noindent - If $C_{s}=C_{s'}$, we have:
        \vspace{-.05in}
		      \begin{eqnarray*}
			& &\left(\bigg( t \otimes \displaystyle \bigotimes_{C \in \pi} \alpha_{\smop{C}} \bigg) \diamondsuit_{s} \bigg( u \otimes \displaystyle \bigotimes_{B \in \rho} \beta_{\smop{B}} \bigg)\right) \diamondsuit_{s'} \Bigg( v \otimes \bigotimes_{A\in \tau} \gamma_{\smop{A}} \Bigg)\\
                &=& \left[\sum_{H \in \smop{sh}\big(\smop{In}(C_s,t),\,\smop{In}(B_r,u)\big)}(t \circ_{\smop{C}_s}^{\smop{H}} u) \otimes (\alpha_{\smop{C}_s} \circ_s \beta_{\smop{B}_r}) \otimes 
   \bigg(\bigotimes_{C'\in \pi \setminus\{C_s\}} \alpha_{\smop{C}'}\bigg)   \otimes \bigg( \bigotimes_{B'\in \rho \setminus_{\{B_r\}}} \beta_{\smop{B}'}\bigg)\right]\\
   &&\diamondsuit_{s'} \Bigg( v \otimes \bigotimes_{A\in \tau} \gamma_{\smop{A}} \Bigg)\\
   &=&\sum_{K \in \smop{sh}\big(\smop{In}(B_r,\,t \circ_{\smop{C}_s}^{\smop{H}} u),\,\smop{In}(B_r,u)\big)} \ \sum_{H \in \smop{sh}\big(\smop{In}(C_s,t),\,\smop{In}(B_r,u)\big)}\Big((t\circ_{\smop{C}_{s}}^{\smop{H}}u)\circ_{\smop{B}_{r}}^{\smop{K}} v \Big) \otimes \Big(\bigotimes_{C \in  \pi\setminus C_{s}}\alpha_{c}\Big) \\
   &&\otimes \Big((\alpha_{\smop{C}_{s}} \circ_s \beta_{\smop{B}_r})\circ_{s'} \gamma_{A_r}\Big)\otimes \Big(\bigotimes_{C \in  \rho \setminus_{B_r}}\beta_{\smop{B}}\Big) \otimes \Big(\bigotimes_{A\in \tau \setminus A_{r}} \gamma_{\smop{A}}\Big)\\
   &&\hbox{(using the fact that $s'$ is an element of the new block $C_s\sqcup_s B_r$)}\\
			&=&\sum_{H \in \smop{sh}\big(\smop{In}(C_s,t),\,\smop{In}(A_r,u\circ_{\smop{B}_{r}}^{\smop{K}} v)\big)} \ \sum_{K \in \smop{sh}\big(\smop{In}(B_r,u),\,\smop{In}(A_r,v)\big)}\Big(t\circ_{\smop{C}_{s}}^{\smop{H}}(u\circ_{\smop{B}_{r}}^{\smop{K}} v) \Big) \otimes \Big(\bigotimes_{C \in  \pi\setminus C_{s}}\alpha_{\smop{C}}\Big) \\
            & & \otimes \Big((\alpha_{\smop{C}_{s}} \circ_{s'} \gamma_{\smop{A}_r} )\circ_s \beta_{\smop{B}_r}\Big)\otimes \Big(\bigotimes_{C \in  \rho \setminus {B_r}}\beta_{\smop{B}}\Big) \otimes \Big(\bigotimes_{A\in \tau \setminus A_{r}} \gamma_{\smop{A}}\Big) \\
			& & \text{(via nested associativity of $\mop{\sc Mag}$ and parallel associativity of q)}\\
			&=&\left(\bigg( t \otimes \displaystyle \bigotimes_{C \in \pi} \alpha_{\smop{C}} \bigg) \diamondsuit_{s'} \bigg( v \otimes \bigotimes_{A\in \tau} \gamma_{\smop{A}} \bigg)\right) \diamondsuit_{s} \Bigg( u \otimes \displaystyle \bigotimes_{B \in \rho} \beta_{\smop{B}} \Bigg).
                \end{eqnarray*}
                \noindent - If $ C_s \ne  C_{s'}$, we have:
			\begin{eqnarray*}
				& &\left(\bigg( t \otimes \displaystyle \bigotimes_{C \in \pi} \alpha_{\smop{C}} \bigg) \diamondsuit_{s} \bigg( u \otimes \displaystyle \bigotimes_{B \in \rho} \beta_{\smop{B}} \bigg)\right) \diamondsuit_{s'} \Bigg( v \otimes \bigotimes_{A\in \tau} \gamma_{\smop{A}} \Bigg)\\
                &=& \left[\sum_{H \in \smop{sh}\big(\smop{In}(C_s,t),\,\smop{In}(B_r,u)\big)}(t \circ_{\smop{C}_s}^{\smop{H}} u) \otimes (\alpha_{\smop{C}_s} \circ_s \beta_{\smop{B}_r}) \otimes 
                \bigg(\bigotimes_{C'\in \pi \setminus\{C_s\}} \alpha_{\smop{C}'}\bigg)   \otimes \bigg( \bigotimes_{B'\in \rho \setminus_{\{B_r\}}} \beta_{\smop{B}'}\bigg)\right]\\
                &&\diamondsuit_{s'} \Bigg( v \otimes \bigotimes_{A\in \tau} \gamma_{\smop{A}} \Bigg)\\
                &=& \sum_{H' \in \smop{sh}\big(\smop{In}(C_s',\,t \circ_{\smop{C}_s}^{\smop{H}} u),\,\smop{In}(A_r,v)\big)} \ \sum_{H \in \smop{sh}\big(\smop{In}(C_s,t),\,\smop{In}(B_r,u)\big)}\bigg((t \circ_{\smop{C}_s}^{\smop{H}} u)\circ_{\smop{C}_{s'}}^{\smop{H}'} v \bigg) \otimes (\alpha_{\smop{C}_s} \circ_s \beta_{\smop{B}_r})\\
                &&\otimes (\alpha_{\smop{C}_{s'}} \circ_{s'} \gamma_{\smop{A}_r}) \otimes \bigg(\bigotimes_{C'\in \pi \setminus\{C_s,C_{s'} \}} \alpha_{\smop{C}'}\bigg) \otimes \bigg( \bigotimes_{B'\in \rho \setminus_{\{B_r\}}} \beta_{\smop{B}'}\bigg) \Bigg( v \otimes \bigotimes_{A\in \tau} \gamma_{\smop{A}} \Bigg)\\
                &=&  \sum_{H \in \smop{sh}\big(\smop{In}(C_s,\,t\circ_{\smop{C}_{s'}}^{\smop{H}'} v),\,\smop{In}(B_r,u)\big)}\ \sum_{H' \in \smop{sh}\big(\smop{In}(C_s',t ),\,\smop{In}(A_r,v)\big)} \bigg((t\circ_{\smop{C}_{s'}}^{\smop{H}'} v )\circ_{\smop{C}_s}^{\smop{H}} u \bigg) \otimes (\alpha_{\smop{C}_s} \circ_s \beta_{\smop{B}_r})\\
                &&\otimes (\alpha_{\smop{C}_{s'}} \circ_{s'} \gamma_{\smop{A}_r}) \otimes \bigg(\bigotimes_{C'\in \pi \setminus\{C_s,C_{s'} \}} \alpha_{\smop{C}'}\bigg) \otimes \bigg( \bigotimes_{B'\in \rho \setminus_{\{B_r\}}} \beta_{\smop{B}'}\bigg) \Bigg( v \otimes \bigotimes_{A\in \tau} \gamma_{\smop{A}} \Bigg)\\
                & & \text{(via parallel associativity of $\mop{\sc Mag }$)}\\
			&=&\left(\bigg( t \otimes \displaystyle \bigotimes_{C \in \pi} \alpha_{\smop{C}} \bigg) \diamondsuit_{s'} \bigg( v \otimes \bigotimes_{A\in \tau} \gamma_{\smop{A}} \bigg)\right) \diamondsuit_{s} \Bigg( u \otimes \displaystyle \bigotimes_{B \in \rho} \beta_{\smop{B}} \Bigg).
                \end{eqnarray*}
                \item Nested associativity: Two subcases again arise.\\
		\noindent - If $t \in B_{r}$, we have:
            \begin{eqnarray*}
			& &\left(\bigg( t \otimes \displaystyle \bigotimes_{C \in \pi} \alpha_{\smop{C}} \bigg) \diamondsuit_{s} \bigg( u \otimes \displaystyle \bigotimes_{B \in \rho} \beta_{\smop{B}} \bigg)\right) \diamondsuit_{t} \Bigg( v \otimes \bigotimes_{A\in \tau} \gamma_{\smop{A}} \Bigg)\\
                &=& \left[\sum_{H \in \smop{sh}\big(\smop{In}(C_s,t),\,\smop{In}(B_r,u)\big)}(t \circ_{\smop{C}_s}^{\smop{H}} u) \otimes (\alpha_{\smop{C}_s} \circ_s \beta_{\smop{B}_r}) \otimes 
            \bigg(\bigotimes_{C'\in \pi \setminus\{C_s\}} \alpha_{\smop{C}'}\bigg)   \otimes \bigg( \bigotimes_{B'\in \rho  \setminus_{\{B_r\}}} \beta_{\smop{B}'}\bigg)\right]\\
            &&\diamondsuit_{t} \Bigg( v \otimes \bigotimes_{A\in \tau} \gamma_{\smop{A}} \Bigg)
            \end{eqnarray*}
            then,
            \begin{eqnarray*}
            & &\left(\bigg( t \otimes \displaystyle \bigotimes_{C \in \pi} \alpha_{\smop{C}} \bigg) \diamondsuit_{s} \bigg( u \otimes \displaystyle \bigotimes_{B \in \rho} \beta_{\smop{B}} \bigg)\right) \diamondsuit_{t} \Bigg( v \otimes \bigotimes_{A\in \tau} \gamma_{\smop{A}} \Bigg)\\
            &=&\sum_{L \in \smop{sh}\big(\smop{In}(B_r, \,t\circ_{\smop{C}_{s}}^{\smop{H}}u),\,\smop{In}(A_r,v)\big)} \ \sum_{H \in \smop{sh}\big(\smop{In}(C_s,t),\,\smop{In}(B_r,u)\big)}\Big((t\circ_{\smop{C}_{s}}^{\smop{H}}u)\circ_{\smop{B}_{r}}^{\smop{L}} v \Big) \otimes \Big(\bigotimes_{C \in  \pi\setminus C_{s}}\alpha_{c}\Big) \\
            &&\otimes \Big((\alpha_{\smop{C}_{s}} \circ_s \beta_{\smop{B}_r})\circ_{t} \gamma_{A_r}\Big)\otimes \Big(\bigotimes_{C \in  \rho \setminus_{B_r}}\beta_{\smop{B}}\Big) \otimes \Big(\bigotimes_{A\in \tau \setminus A_{r}} \gamma_{\smop{A}}\Big)\\
            &=&\sum_{H \in \smop{sh}\big(\smop{In}(C_s,t),\,\smop{In}(A_r,\,u\circ_{\smop{B}_{r}}^{\smop{L}} v )\big)} \ \sum_{L \in \smop{sh}\big(\smop{In}(B_r, u),\,\smop{In}(A_r,v)\big)} \Big(t\circ_{\smop{C}_{s}}^{\smop{H}}(u\circ_{\smop{B}_{r}}^{\smop{L}} v )\Big) \otimes \Big(\bigotimes_{C \in  \pi\setminus C_{s}}\alpha_{c}\Big) \\
        &&\otimes \Big(\alpha_{\smop{C}_{s}} \circ_s (\beta_{\smop{B}_r}\circ_{t} \gamma_{A_r})\Big)\otimes \Big(\bigotimes_{C \in  \rho \setminus_{B_r}}\beta_{\smop{B}}\Big) \otimes \Big(\bigotimes_{A\in \tau \setminus A_{r}} \gamma_{\smop{A}}\Big)\\
        & & \text{(via nested associativity of $\mop{\sc Mag}$ and q)}\\
        &=&\bigg( t \otimes \displaystyle \bigotimes_{C \in \pi} \alpha_{\smop{C}} \bigg) \diamondsuit_{s} \left(\bigg( u \otimes \displaystyle \bigotimes_{B \in \rho} \beta_{\smop{B}} \bigg) \diamondsuit_{t} \Bigg( v \otimes \bigotimes_{A\in \tau} \gamma_{\smop{A}} \Bigg)\right)\\
             \end{eqnarray*}
             \noindent - If $t \in B_{t} \ne B_{r}$, we have:
            \begin{eqnarray*}
			& &\left(\bigg( t \otimes \displaystyle \bigotimes_{C \in \pi} \alpha_{\smop{C}} \bigg) \diamondsuit_{s} \bigg( u \otimes \displaystyle \bigotimes_{B \in \rho} \beta_{\smop{B}} \bigg)\right) \diamondsuit_{t} \Bigg( v \otimes \bigotimes_{A\in \tau} \gamma_{\smop{A}} \Bigg)\\
             &=& \left[\sum_{H \in \smop{sh}\big(\smop{In}(C_s,t),\,\smop{In}(B_r,u)\big)}(t \circ_{\smop{C}_s}^{\smop{H}} u) \otimes (\alpha_{\smop{C}_s} \circ_s \beta_{\smop{B}_r}) \otimes 
            \bigg(\bigotimes_{C'\in \pi \setminus\{C_s\}} \alpha_{\smop{C}'}\bigg)   \otimes \bigg( \bigotimes_{B'\in \rho  \setminus_{\{B_r\}}} \beta_{\smop{B}'}\bigg)\right]\\
            &&\diamondsuit_{t} \Bigg( v \otimes \bigotimes_{A\in \tau} \gamma_{\smop{A}} \Bigg)\\
            &=&\sum_{L \in \smop{sh}\big(\smop{In}(B_t, \,t\circ_{\smop{C}_{s}}^{\smop{H}}u),\smop{In}(A_r,v)\big)} \ \sum_{H \in \smop{sh}\big(\smop{In}(C_s,t),\,\smop{In}(B_r,u)\big)}\Big((t\circ_{\smop{C}_{s}}^{\smop{H}}u)\circ_{\smop{B}_{t}}^{\smop{L}} v \Big) \otimes \Big(\bigotimes_{C \in  \pi\setminus C_{s}}\alpha_{c}\Big) \\
            &&\otimes (\alpha_{\smop{C}_{s}} \circ_s \beta_{\smop{B}_r}) \otimes (\beta_{\smop{B}_t} \circ_{t} \gamma_{A_r})\otimes \Big(\bigotimes_{C \in  \rho \setminus{\{B_r, B_t  \}}}\beta_{\smop{B}}\Big) \otimes \Big(\bigotimes_{A\in \tau \setminus A_{r}} \gamma_{\smop{A}}\Big)
            \end{eqnarray*}
            then 
            \begin{eqnarray*}
            & &\left(\bigg( t \otimes \displaystyle \bigotimes_{C \in \pi} \alpha_{\smop{C}} \bigg) \diamondsuit_{s} \bigg( u \otimes \displaystyle \bigotimes_{B \in \rho} \beta_{\smop{B}} \bigg)\right) \diamondsuit_{t} \Bigg( v \otimes \bigotimes_{A\in \tau} \gamma_{\smop{A}} \Bigg)\\
            &=&\sum_{H \in \smop{sh}\big(\smop{In}(C_s,t),\,\smop{In}(A_r,u\circ_{\smop{B}_{t}}^{\smop{L}} v )\big)} \ \sum_{L \in \smop{sh}\big(\smop{In}(B_t, u),\,\smop{In}(A_r,v)\big)} \Big(t\circ_{\smop{C}_{s}}^{\smop{H}}(u\circ_{\smop{B}_{t}}^{\smop{L}} v )\Big) \otimes \Big(\bigotimes_{C \in  \pi\setminus C_{s}}\alpha_{c}\Big) \\
        &&\otimes (\alpha_{\smop{C}_{s}} \circ_s \beta_{\smop{B}_r}) \otimes (\beta_{\smop{B}_t} \circ_{t} \gamma_{A_r})\otimes \Big(\bigotimes_{C \in  \rho \setminus\{B_r,B_t\}}\beta_{\smop{B}}\Big) \otimes \Big(\bigotimes_{A\in \tau \setminus A_{r}} \gamma_{\smop{A}}\Big)\\
        & & \text{(via nested associativity of $\mop{\sc Mag}$)}\\
        &=&\bigg( t \otimes \displaystyle \bigotimes_{C \in \pi} \alpha_{\smop{C}} \bigg) \diamondsuit_{s} \left(\bigg( u \otimes \displaystyle \bigotimes_{B \in \rho} \beta_{\smop{B}} \bigg) \diamondsuit_{t} \Bigg( v \otimes \bigotimes_{A\in \tau} \gamma_{\smop{A}} \Bigg)\right)\\
        \end{eqnarray*}
            \end{itemize}
\end{proof}

 \end{document}